\documentclass[a4paper, leqno, 11pt]{amsart}
\usepackage[utf8]{inputenc}
\usepackage[a4paper, left=90pt, right=90pt, bottom=95pt, top=95pt]{geometry}

\usepackage{amsthm}
\usepackage{amssymb}
\usepackage[american]{babel}
\usepackage{exscale}
\usepackage[mathscr]{euscript}
\usepackage{url}
\usepackage[all,ps,tips,tpic]{xy}
\usepackage{graphicx}
\usepackage{color}
\usepackage{verbatim}

\usepackage{tikz} 
\usetikzlibrary{matrix,arrows}

\usepackage{amscd}

\numberwithin{equation}{section}
\addtocounter{section}{-1}

\DeclareMathOperator{\GU}{GU}

\DeclareMathOperator{\SL}{SL}
\DeclareMathOperator{\SU}{SU}
\DeclareMathOperator{\SP}{Sp}

\DeclareMathOperator{\red}{red}

\DeclareMathOperator{\Spec}{Spec}
\DeclareMathOperator{\str}{-}

\DeclareMathOperator{\Aff}{Aff}
\DeclareMathOperator{\aff}{af}
\DeclareMathOperator{\drv}{der}

\DeclareMathOperator{\alg}{alg}
\DeclareMathOperator{\scon}{sc}
\DeclareMathOperator{\chara}{char}
\DeclareMathOperator{\anti}{anti}
\DeclareMathOperator{\Max}{Max}
\DeclareMathOperator{\Min}{Min}

\DeclareMathOperator{\Gal}{Gal}
\DeclareMathOperator{\define}{\hspace{0.1cm}\stackrel{\text{\tiny def}}{=}\hspace{0.1cm}}

\DeclareMathOperator{\fin}{fin}

\DeclareMathOperator{\spanned}{span}
\DeclareMathOperator{\naive}{naive}
\DeclareMathOperator{\loc}{loc}

\DeclareMathOperator{\val}{val}
\DeclareMathOperator{\ort}{ort}
\DeclareMathOperator{\sympl}{spl}
\DeclareMathOperator{\inv}{inv}

\newcommand{\pot}[1]{ [\hspace{-0,5mm}[ {#1} ]\hspace{-0,5mm}] }
\newcommand{\rpot}[1]{ (\hspace{-0,7mm}( {#1} )\hspace{-0,7mm}) }
\newcommand{\iweyl}[2]{{_{#1}\hspace{-0.01cm}\widetilde{W}}\vspace{-0.01cm}^{#2}}

\newcommand{§}{\ss}

\newtheorem{lem}{Lemma}[section]

\newtheorem{cor}[lem]{Corollary}
\newtheorem{thm}[lem]{Theorem}
\newtheorem{prop}[lem]{Proposition}

\newtheorem*{acknowledge}{Acknowledgements}
\newtheorem*{notation}{Notation}

\theoremstyle{definition}
\newtheorem{definition}[lem]{Definition}

\theoremstyle{remark}
\newtheorem{rem}[lem]{Remark}

\newtheorem{examples}[lem]{Examples}

\title[Schubert varieties and local models]{Schubert varieties in twisted affine flag varieties and local models}
\author[T. Richarz]{by Timo Richarz}
\begin{document}

\maketitle

\begin{abstract}
Our concern in this paper is the dimension and inclusion relations of Schubert varieties in twisted partial affine flag varieties. In the end we apply our results to some local models of certain Schubert varieties.
\end{abstract}

\section{Introduction} \label{intro}
Let $G$ be a connected reductive linear algebraic group over the local field $K=k\rpot{t}$ of Laurent series with algebraically closed residue field $k$ and ring of integers $\mathcal{O}_K=k\pot{t}$. The loop group $LG$ is the functor
\[
LG:\Spec(R) \longmapsto G(\Spec(R\rpot{t}))
\]
on the category of affine $k$-schemes. This functor is representable by an ind-scheme of ind-finite type (=inductive limit of schemes of finite type). To a facet $F$ in the Bruhat-Tits building of $G$ there is associated a unique smooth affine group scheme $\mathcal{P}_F$ with connected fibers over $\mathcal{O}_K$ such that its generic fiber is $G$, and its $\mathcal{O}_K$-valued points $\mathcal{P}_F(\mathcal{O}_K)$ are the parahoric subgroup of $G(K)$ attached to $F$. To $\mathcal{P}_F$ corresponds an infinite-dimensional affine group scheme $L^+\mathcal{P}_F$ over $k$ with
\[
L^+\mathcal{P}_F(\Spec(R))=\mathcal{P}_F(\Spec(R\pot{t})).
\]
The quotient $\mathcal{F}_F=LG/L^+\mathcal{P}_F$ (in the sense of fpqc-sheaves) is called the \emph{twisted} affine flag variety associated to $G$ and $F$ and is representable by an ind-proj scheme (=inductive limit of projective schemes) over $k$ (cf. \cite{PRtwisted}). If $F$ is a special vertex, we also call $\mathcal{F}_F$ the \emph{twisted} affine Gra§mannian.
We consider Schubert varieties in $\mathcal{F}_F$, i.e. reduced closures of $L^+\mathcal{P}_F$-orbits in $\mathcal{F}_F$. These are finite-dimensional projective varieties over $k$. Our concern in this paper is the dimension and inclusion relations of Schubert varieties. In the end we apply our results to some local models of certain Schubert varieties.

Fix a maximal $K$-split torus $S$ in $G$ whose apartment contains $F$. The Schubert varieties in $\mathcal{F}_F$ are enumerated by double cosets $\widetilde{W}_F\backslash\widetilde{W}/\widetilde{W}_F$ of the Iwahori-Weyl group $\widetilde{W}$ modulo the subgroup $\widetilde{W}_F$ attached to $F$. For $w\in\widetilde{W}_F\backslash\widetilde{W}/\widetilde{W}_F$, the corresponding Schubert variety $S_w$ is the closure of the $L^+\mathcal{P}_F$-orbit
\[
(L^+\mathcal{P}_F\;w\;L^+\mathcal{P}_F)/L^+\mathcal{P}_F\subset\mathcal{F}_F
\]
equipped with the reduced scheme structure. \\ 
The choice of an alcove containing $F$ in its closure endows the group $\widetilde{W}$ with the structure of a quasi Coxeter-system, which may thus be equipped with a Bruhat-Chevalley (partial) order $\leq$ and a length function $l$. To $w\in\widetilde{W}_F\backslash\widetilde{W}/\widetilde{W}_F$, we associate a unique representative ${_Fw}^F$ in $\widetilde{W}$ which satisfies
\[
l({_Fw}^F)\;\:=\:\underset{w_1\in\widetilde{W}_F}{\Max}\;\underset{w_2\in\widetilde{W}_F}{\Min}l(w_1\dot{w}w_2),
\]
where $\dot{w}$ is any representative of $w$ in $\widetilde{W}$. If we set $\iweyl{F}{F}\define\{w\in\widetilde{W}\;|\;w={_F\hspace{-0.01cm}w}^F\}$, then the Schubert varieties in $\mathcal{F}_F$ are also enumerated by $\iweyl{F}{F}$.

\begin{prop}
Let $S_w$ be the Schubert variety in $\mathcal{F}_F$ corresponding to $w\in\iweyl{F}{F}$. \\ \smallskip
\emph{(i)} $S_w$ is set theoretically a disjoint union of locally closed strata
\[
S_w=\coprod_{\substack{v\in\iweyl{F}{F} \\  v\leq w}}(L^+\mathcal{P}_F\;v\;L^+\mathcal{P}_F)/L^+\mathcal{P}_F.
\]
\emph{(ii)} The dimension of $S_w$ is equal to $l(w)$. 
\end{prop}

\noindent Let $W_0$ be the relative Weyl group of $G$ with respect to $S$. Denote by $T$ the centralizer of $S$ in $G$. Since $G$ is quasi-split, $T$ is a maximal torus. Let $F=\{x\}$ be a special vertex and $I$ be the absolute Galois group. Then it turns out (cf. Cor. \ref{specialreps}) that $\iweyl{F}{F}=X_*(T)_I^-$ are the antidominant representatives of the $W_0$-orbits in the coinvariants under $I$ of the geometric cocharacters $X_*(T)$. Let $G_{\drv}$ be the derived group of $G$ with simply connected covering $G_{\scon}\rightarrow G_{\drv}$. Denote by $T_{\scon}$ the inverse image of $G_{\drv}\cap T$ in $G_{\scon}$. Let $\Sigma_0$ be the unique reduced root system such that
\[
X_*(T_{\scon})_I\rtimes W_0 = Q(\Sigma_0^{\vee})\rtimes W(\Sigma_0),
\]
compatible with the semidirect product decomposition. Here $W(\Sigma_0)$ is the Weyl group, and $Q(\Sigma_0^{\vee})$ is the coroot lattice with respect to $\Sigma_0$. 

\begin{cor}\label{dimcor}
Let $S_{\mu}$ be the Schubert variety in $\mathcal{F}_{F}$ corresponding to $\mu\in X_*(T)_I^-$. Then
\[
\dim(S_{\mu})=|\langle\mu,2\rho\rangle|,
\]
where $\rho$ denotes the halfsum of the positive roots in $\Sigma_0$. 
\end{cor}

\begin{rem}
The proposition and corollary are well-known in the case of split groups.
\end{rem}

Schubert varieties admit a good class of resolutions of singularities. The following Theorem is the analogue in the present case of an analogous resolution in the case of Schubert varieties in finite dimensional flag varieties (cf. \cite[Ch. 9.1]{BLak}). A proof in the case of simply connected split groups was explained to me by N. Perrin, and his arguments extend to the general case.

\begin{thm}\label{resolutionthm}
Let $F$ be a facet in the Bruhat-Tits building of $G$, and let $S$ be a fixed Schubert variety in $\mathcal{F}_F$. Assume for simplicity that $S$ is contained in the neutral component of $\mathcal{F}_F$. Then there exist parahoric group schemes $\mathcal{P}_1$,...,$\mathcal{P}_n$ and $\mathcal{Q}_1$,...,$\mathcal{Q}_{n}$ with 
\begin{align*}
& \emph{(i)} && L^+\!\mathcal{Q}_i\subset L^+\mathcal{P}_i\cap L^+\mathcal{P}_{i+1} \quad \text{for}\;\; i=1,...,n-1, \\
& \emph{(ii)} && L^+\!\mathcal{Q}_n\subset L^+\mathcal{P}_n\cap L^+\mathcal{P}_F \\
& \emph{(iii)} && L^+\mathcal{P}_F\subset L^+\mathcal{P}_1
\end{align*}
such that the morphism 
\[
L^+\mathcal{P}_1 \times^{L^+\!\mathcal{Q}_1}...\times^{L^+\!\mathcal{Q}_{n-1}} L^+\mathcal{P}_n/L^+\!\mathcal{Q}_n \longrightarrow \mathcal{F}_F
\]
given by multiplication factors through $S$ and induces a birational and $L^+\mathcal{P}_1$-equivariant morphism 
\[
L^+\mathcal{P}_1 \times^{L^+\!\mathcal{Q}_1}...\times^{L^+\!\mathcal{Q}_{n-1}} L^+\mathcal{P}_n/L^+\!\mathcal{Q}_n\to S. 
\]
Moreover, the source of this morphism is an iterated extension of homogeneous spaces and is hence smooth over $k$.
\end{thm}

\noindent Suppose that $G$ splits over $K$. If $F=C$ is an alcove, then a resolution like the one in Theorem \ref{resolutionthm} can be provided by the Demazure resolution. The case of Schubert varieties in finite-dimensional flag varieties is treated in \cite{Demazure}, and the affine case is an extension of these arguments \cite[\S 8.c.]{PRtwisted}. If $F$ is a hyperspecial vertex and if $w$ is a quasi-minuscule coweight, then the resolution of Theorem $\ref{resolutionthm}$ is the same as the one constructed by B.-C. Ng\^o and P. Polo in \cite[Lemme 7.3]{NgoPolo}.

Theorem \ref{resolutionthm} has applications to the theory of local models of Shimura varieties with parahoric level structure. We illustrate this by the following example. Let $G=\GU(W,\phi)$ be the group of unitary similitudes for a hermitian vector space $(W,\phi)$ of odd dimension $n=2m+1\geq3$ over a totally ramified extension $E$ of $\mathbb{Q}_p$ with $p\neq 2$. Let $P\subset G(\mathbb{Q}_p)$ a parahoric subgroup with corresponding parahoric group scheme $\mathcal{P}$ over $\mathbb{Z}_p$. G. Pappas and M. Rapoport \cite{PRunitary} attach to $P$ the \emph{naive} local model $M^{\naive}_P$ over $\mathcal{O}_E$ which models the \'etale local structure of a corresponding model for the unitary Shimura variety of signature $(r,s)$ at a given ramified prime $p$. The local model $M^{\naive}_P$ is a projective $\mathcal{O}_E$ scheme which carries an action of $\mathcal{P}\otimes\mathcal{O}_E$ and whose generic fiber is of dimension $rs$. It admits a moduli description, but fails to be flat in general \cite{Pappas}. To remedy the failure of flatness, Pappas adds the $\wedge$-condition to the moduli description of $M^{\naive}_P$, thus defining a closed subscheme, $M^{\wedge}_P\subset M^{\naive}_P$, the \emph{wedge} local model, which agrees with $M^{\naive}_P$ on generic fibers. B. Smithling \cite{Smithling} shows that the scheme $M^{\wedge}_P$ is topologically flat. There is a further variant 
\[
M^{\loc}_P\subset M^{\wedge}_P\subset M^{\naive}_P,
\] 
which is by definition the flat closure of the generic fiber in $M^{\naive}_P$, and is thus itself flat. It is called the \emph{local model} and conjecturally admits a moduli description \cite[Conjecture 7.3]{PRunitary}. \\
An important tool for the study of these local models is the embedding of the geometric special fiber $\overline{M}^{\naive}_P=M^{\naive}_P\otimes\mathbb{F}_p^{\alg}$ of the naive local model in the partial twisted affine flag variety $\mathcal{F}_P$ corresponding to $P$ (cf. \cite[\S 3.c]{PRunitary}). This embedding is equivariant under the action of $L^+\mathcal{P}$.\\
Now assume that $P$ is a special parahoric subgroup. There are two conjugacy classes of special parahoric subgroups, and we assume that $P$ corresponds to the class given by $I=\{m\}$ in the notation of \cite[1.b.3 a)]{PRunitary}. So $\mathcal{F}=\mathcal{F}_P$ is a twisted affine Gra§mannian. Let $M^{\wedge}_s=M^{\wedge}_P$ be the wedge local model for signature $(r,s)$. This case is also considered by K. Arzdorf in \cite{Arzdorf}.

\begin{thm}\label{locmodthm}
There is a projective morphism 
$
\pi_s:\mathcal{M}_s^{\wedge}\longrightarrow M^{\wedge}_s
$
of $\mathcal{O}_E$-schemes which is the identity on generic fibers, and which satisfies the following properties:\smallskip \\
 \emph{(i)} The geometric special fiber
 $
 \overline{\mathcal{M}}^{\wedge}_s=\cup_{i=0}^sZ_i
 $
 is the union of $s+1$ irreducible and generically smooth divisors on the scheme $\mathcal{M}^{\wedge}_s$ which are the strict transforms of $s+1$ linearly ordered strata in $\overline{M}^{\wedge}_s$. The divisors $Z_0$, $Z_s$ are smooth, and the restriction $\overline{\pi}_s|_{Z_s}:Z_s\rightarrow \overline{M}^{\wedge}_s$ is a surjective birational projective morphism. In particular, $\overline{M}^{\wedge}_s$ is irreducible and contains a non-empty open reduced subscheme. \smallskip \\
 \emph{(ii)} The special fiber of the local model $\overline{M}^{\loc}_s$ is equal to the reduced locus $(\overline{M}^{\wedge}_s)_{\text{\rm red}}$, and hence $\overline{\pi}_s|_{Z_s}$ factors through $\overline{M}^{\loc}_s$, thus defining the morphism $\theta:Z_s\rightarrow \overline{M}^{\loc}_s$ which is embedded in a twisted affine Gra§mannian, such that $\theta$ is identical to an equivariant affine Demazure resolution in the sense of Theorem \ref{resolutionthm}.
\end{thm}

\noindent As a corollary we obtain the main result of K. Arzdorf \cite{Arzdorf}.

\begin{cor}
The geometric special fiber of the local model $M_s^{\loc}$ is normal, Frobenius-split and has only rational singularities.
\end{cor} 

\begin{rem}
 If $(r,s)=(n-1,1)$, the morphism $\pi_1:\mathcal{M}_1^{\wedge}\to M_1^{\wedge}$ is a semistable resolution, i.e. $\mathcal{M}^{\wedge}_1$ is regular and the irreducible components of the special fiber are smooth divisors crossing normally. In fact, in this case the scheme $\overline{M}_1^{\loc}$ is even smooth, although it is stratified by $s+1=2$ orbits as a Schubert variety in the twisted affine Gra§mannian $\mathcal{F}$ as was observed by the author \cite[Prop. 4.16]{Arzdorf}. As a consequence the analogue of the Theorem of S. Evens and I. Mirkovi\'c \cite[Corollary B]{EvensMirkovic} that a Schubert variety in the affine Gra§mannian is singular along its boundary, fails for \emph{twisted} affine Gra§mannians.   
\end{rem}

\begin{acknowledge} \rm First of all I thank my advisor M. Rapoport for his steady encouragement and his interest in my work. I am also grateful to N. Perrin for his explanations about the construction of an equivariant Demazure resolution. Moreover, I warmly thank J.-L. Waldspurger for his explanations about Bruhat-Tits theory and some related combinatorial problems. 
\end{acknowledge}

\begin{notation}
\rm Let $K$ be a discretely valued complete field whose valuation $\val:K^{\times}\to\mathbb{Z}$ is non-trivial, non-archimedean and normalized such that uniformizers have valuation $1$. Denote by $\mathcal{O}_K$ its ring of integers with maximal ideal $\mathfrak{m}_K$ and residue field $k=\mathcal{O}_K/\mathfrak{m}_K$. \\
We assume that $k$ is \emph{algebraically closed}. Let $K^{\text{sep}}$ be a separable closure of $K$.
\end{notation}

\section{The Iwahori-Weyl group}  
First we recall some facts on the Iwahori-Weyl group as given by T. Haines and M. Rapoport in \cite{HR} and prove some combinatorial lemmas needed later.

Let $G$ be a connected reductive linear algebraic group over $K$. Note that $G$ is quasi-split by Steinberg's Theorem. Let $\mathscr{B}=\mathscr{B}(G,K)$ be the (enlarged) Bruhat-Tits building.

Fix a maximal $K$-split torus $S$. Let $T$ be the centralizer (a maximal torus) of $S$, and let $N$ be the normalizer of $S$. Denote by $W_0=N(K)/T(K)$ the relative Weyl group of $G$ with respect to $S$ and denote by $\mathscr{A}=\mathscr{A}(G,S,K)$ the apartment of $\mathscr{B}$ corresponding to $S$. 

\begin{definition}\label{IwahoriWeyl}
(i) The \emph{Iwahori-Weyl group} $\widetilde{W}(G,S)$ of $G$ with respect to $S$ is 
\[
\widetilde{W}=\widetilde{W}(G,S)\define N(K)/\mathcal{T}^0(\mathcal{O}_K),
\]
where $\mathcal{T}^0$ is the connected N\'eron model of $T$ over $\mathcal{O}_K$. \smallskip \\
(ii) Let $F$ be a facet in $\mathscr{A}$ with corresponding parahoric subgroup $P_F$. The subgroup $\widetilde{W}_F$ of the Iwahori-Weyl group corresponding to $F$ is
\[
\widetilde{W}_F\define N(K)\cap P_F/\mathcal{T}^0(\mathcal{O}_K).
\]
\end{definition}

\begin{rem} \label{normalization}
Let $\pi_1(G)$ be the quotient of $X_*(T)$ by the coroot lattice of $G$ with respect to $T$. The absolute Galois group $I=\Gal(K^{\text{sep}}/K)$ acts on $\pi_1(G)$, and we denote by $\pi_1(G)_I$ the coinvariants under this action. Kottwitz defines in \cite[\S 7]{Kott} a surjective morphism of groups  
\begin{equation}
\kappa_G: G(K) \longrightarrow \pi_1(G)_I. \label{kottwitz}
\end{equation}
We choose a different normalization of $\kappa_G$. For the torus $T$ our normalization differs from Kottwitz' normalization by a sign, i.e. we demand $q_T\circ \kappa_T=-v_T$ in Kottwitz' notation (\emph{loc. cit.}).
\end{rem}

The group $N(K)$ operates on $\mathscr{A}$ by affine transformations
\[
\nu: N(K)\longrightarrow \Aff(\mathscr{A}).
\] 
The kernel $\ker(\nu)$ contains the group $\mathcal{T}^0(\mathcal{O}_K)$. Hence, we obtain an action of $\widetilde{W}$ on $\mathscr{A}$. Let $G_1$ be the kernel of the Kotwitz morphism $\kappa_G$ and set $N_1=N(K)\cap G_1$. Fix an alcove $C$ in the apartment $\mathscr{A}$, and denote the corresponding Iwahori subgroup by $B$. Let $\mathbb{S}$ be the set of simple reflections at the walls of $C$. By Bruhat and Tits \cite[Prop. 5.2.12]{BT2}, the quadruple
\begin{equation}
(G_1,B,N_1,\mathbb{S})  \label{doublesystem}
\end{equation}
is a double Tits system with affine Weyl group\footnote{We used that $B\cap N_1=\mathcal{T}^0(\mathcal{O}_K)$.} $W_{\aff}=N_1/\mathcal{T}^0(\mathcal{O}_K)$, and the inclusion $G_1\subset G(K)$ is $B\text{-}N\text{-}$adapted of connected type (cf. \cite[4.1.3]{BT1}). \\
For any facet $F$ contained in the closure of $C$, the group $\widetilde{W}_F$ is the parabolic subgroup of the Coxeter-system $(W_{\aff},\mathbb{S})$ generated by the reflections at the walls of $C$ which contain $F$. A consequence of \eqref{doublesystem} is the following lemma:

\begin{lem}[\cite{HR} Prop. 8]\label{doubleclasses}
Let $F$ (resp. $F'$) be a facet contained in the closure of $C$, and let $P_F$ (resp. $P_{F'}$) be the associated parahoric subgroup. There is a bijection
\begin{align*}
\widetilde{W}_{F'}\backslash\widetilde{W}/\widetilde{W}_F & \overset{\cong}{\longrightarrow}P_{F'}\backslash G(K)/P_F \\
\widetilde{W}_{F'}\,w\,\widetilde{W}_F & \longmapsto P_{F'}\,n_w\,P_F,
\end{align*}
where $n_w$ denotes a representative of $w$ in $N(K)$.
\end{lem}

\begin{rem}
If $F'=F$ is a special vertex, then $\widetilde{W}_F$ maps (via the canonical projection) isomorphically onto $W_0$ and presents the Iwahori-Weyl group as a semidirect product
\begin{equation}
\widetilde{W}=(T(K)/\mathcal{T}^0(\mathcal{O}_K))\rtimes W_0. \label{specialdecom}
\end{equation}
The group $\mathcal{T}^0(\mathcal{O}_K)$ is equal to the kernel of the Kottwitz morphism $\kappa_T$, and hence we obtain an isomorphism
\[
T(K)/\mathcal{T}^0(\mathcal{O}_K)\xrightarrow{\kappa_T}X_*(T)_I.
\] 
Therefore, the double classes modulo $P_F$ are enumerated by the $W_0$-orbits in $X_*(T)_I$.
\end{rem}

There is an exact sequence of groups
\[
1 \longrightarrow W_{\aff} \longrightarrow \widetilde{W} \longrightarrow N(K)/N_1 \longrightarrow 1,
\] 
and the group $N(K)/N_1$ may be identified with the stabilizer of the alcove $C$ in $\widetilde{W}$. Using the Kottwitz morphism $\kappa_G$ , the inclusion $N(K)\subset G(K)$ gives an isomorphism 
\[
N(K)/N_1\cong G(K)/G_1\xrightarrow{\kappa_G}\pi(G)_I,
\]
and hence a semidirect product decomposition
\begin{equation}
\widetilde{W}=\pi_1(G)_I\ltimes W_{\aff}. \label{semidecom}
\end{equation}
Since $(W_{\aff},\mathbb{S})$ is a Coxeter-system, it is equipped with the Bruhat-Chevalley (partial) order $\leq$ and the length function $l$. By \eqref{semidecom} the Iwahori-Weyl group $\widetilde{W}$ is equipped with a quasi-Coxeter structure. 

\begin{rem}\label{simplyrem}
If $G$ is simply connected, then $G_1=G(K)$ and hence $\widetilde{W}=W_{\aff}$ is a Coxeter group.
\end{rem}

Fix two facets $F'$ and $F$ in the closure of $C$.

\begin{lem}\label{representants}
Let $w\in\widetilde{W}$. \\
\noindent\emph{(i)} There exists a unique element $w^F$ of minimal length in $w\widetilde{W}_F$.\\
\noindent\emph{(ii)}  There exists a unique element ${_{F'}\hspace{-0.01cm}w}^F$ of maximal length in $\{(vw)^F \;|\; v\in\widetilde{W}_{F'}\}$. 
\end{lem}

\begin{proof}
We may assume that $\widetilde{W}=W_{\aff}$. Then the result follows from the proof of Lemma \ref{replength} below.
\end{proof}

Denote by $\iweyl{F'}{F}\subset\widetilde{W}$ the subset
\begin{equation}
\iweyl{F'}{F}\define \{{_{F'}\hspace{-0.01cm}w}^F\;|\;w\in\widetilde{W}\}
\end{equation}
of the Iwahori-Weyl group. This set maps bijectively under the natural projection onto the set of double classes
\begin{equation}
\iweyl{F'}{F}\overset{\cong}{\longrightarrow} \widetilde{W}_{F'}\backslash\widetilde{W}/\widetilde{W}_F. \label{representatives}
\end{equation}

Let $V=V(G,S)$ be the $\mathbb{R}$-vector space $\mathbb{R}\otimes X_*(S)$. We have a natural perfect pairing $\langle\str,\str\rangle:X_*(S)\times X^*(S)\to\mathbb{Z}$ of free $\mathbb{Z}$-modules. Tensoring with $\mathbb{R}$ we get a natural identification of $\mathbb{R}\otimes X^*(S)$ with the dual space $V^{\vee}$ and hence a perfect pairing 
\begin{equation}
\langle\str,\str\rangle:V\times V^{\vee}\longrightarrow\mathbb{R}. \label{pairing}
\end{equation}
Choose a special vertex $\{x\}$ in the closure of $C$. Then we may identify $\mathscr{A}=V$ such that $\{x\}$ is identified with $0$. The Iwahori-Weyl group is represented as a semidirect product $\widetilde{W}=X_*(T)_I\rtimes W_0$. By Bourbaki \cite[Chap. VI, \S 2, n$^o$ 5, Prop. 8]{Bourbaki}, there is a reduced root system $\Sigma_0$ such that $W_{\aff}=Q^{\vee}(\Sigma_0)\rtimes W(\Sigma_0)$ is the affine Weyl group of $\Sigma_0$. Denote by 
\[
\Sigma=\{a+k\;|\;a\in\Sigma_0\;\;\text{and}\;\; k\in\mathbb{Z}\}
\] 
the associated system of affine roots. The Iwahori-Weyl group $\widetilde{W}$ acts on $\Sigma$ by 
\[
w(\alpha)(x)= \alpha(w^{-1}\cdot x) ,
\]
for $x\in V$, $\alpha\in\Sigma$ and $w\in\widetilde{W}$. For a root $\alpha\in\Sigma$, we write $\alpha(F)=0$, if $\alpha$ vanishes on $F$ and $\alpha(F)\geq 0$, if $\alpha$ takes non-negative values on $F$. Let $\Sigma(F)$ be the set of all affine roots with $\alpha(F)=0$. Fix an element $w\in \widetilde{W}$. The length $l(w)$ is equal to the number of elements of 
\begin{equation}
\{\alpha\in\Sigma\;|\;\alpha(C)>0\;\;\text{and}\;\; w(\alpha)(C)<0\}. \label{number}
\end{equation}
Writing $w=e^{\mu}\cdot w_{\fin}$ according to the semidirect product decomposition given by the special vertex $x$, we deduce from \eqref{number} the formula 
\begin{equation}
l(e^{\mu}\cdot w_{\fin})\quad=\sum_{\substack{a(C)>0 \\ w_{\fin}^{-1}(a)(C)>0}}\!\!\!\!\!{|\langle \mu,a\rangle|}  \quad\; +\sum_{\substack{a(C)>0 \\ w_{\fin}^{-1}(a)(C)<0}}\!\!\!\!\!{|\langle\mu,a\rangle-1 |}, \label{numberpairing}
\end{equation}
where $a$ runs through the elements of $\Sigma_0$, and $\langle\str,\str\rangle:V\times V^{\vee}\to\mathbb{R}$ is the natural pairing.

\begin{lem}[J.-L. Waldspurger]\label{replength}
The length $l({_{F'}w}^F)$ is equal to the number of elements of 
\[
\{\alpha\in\Sigma\backslash\Sigma(F)\; | \; \alpha(C)>0 \;\; \text{and} \;\; w(\alpha)(F')\leq 0\}.
\]
\end{lem}

\begin{proof}
First, we show that the length $l(w^F)$ is equal to the number of elements of the set 
\begin{equation}\label{rightmin}
\{\alpha\in\Sigma\backslash\Sigma(F)\; | \; \alpha(C)>0 \;\; \text{and} \;\; w(\alpha)(C)< 0\}.
\end{equation}
Let $u\in \widetilde{W}_F$. The length $l(wu)$ is the number of elements of the set
\[
\{\alpha\in\Sigma\; | \; \alpha(C)>0 \;\; \text{and} \;\; wu(\alpha)(C)< 0\}.
\]
This set decomposes into a disjoint union $X(u)\sqcup Y(u)$ with
\begin{align*}
X(u) &=\{\alpha\in\Sigma\backslash\Sigma(F)\; | \; \alpha(C)>0 \;\; \text{and} \;\; wu(\alpha)(C)< 0\} \\
Y(u)&=\{\alpha\in\Sigma(F)\; | \; \alpha(C)>0 \;\; \text{and} \;\; wu(\alpha)(C)< 0\}.
\end{align*}
The map $\alpha\mapsto u(\alpha)$ preserves the positivity for $\alpha\in\Sigma\backslash\Sigma(F)$. Hence, $X(u)$ is independent of $u$ and equal to $X(1)$. Therefore, it is enough to prove that there is a unique $u\in\widetilde{W}_F$ such that $Y(u)=\varnothing$. This means that  $u$ maps the set $\{\alpha\in\Sigma(F)\;|\;\alpha(C)>0\}$ to the set $\{\alpha\in\Sigma(F)\;|\; w(\alpha)(C)>0\}$. But those sets are Borel subsets of $\Sigma(F)$ and it is well known that there exists a unique $u\in\widetilde{W}_F$ mapping the first set to the second. This proves \eqref{rightmin} and also Lemma \ref{representants} (i). \\
Now let $u\in\widetilde{W}_{F'}$. The length $l((uw)^F)$ is the number of elements of the set
\[
\{\alpha\in\Sigma\backslash\Sigma(F)\; | \; \alpha(C)>0 \;\; \text{and} \;\; uw(\alpha)(C)< 0\}.
\]
This set decomposes into a disjoint union $X(u)\sqcup Y(u)$ with
\begin{align*}
X(u) &=\{\alpha\in\Sigma\backslash\Sigma(F)\; | \; \alpha(C)>0, \;\; uw(\alpha)(C)< 0 \;\; \text{and} \;\; w(\alpha) \not \in \Sigma(F')\} \\
Y(u)&=\{\alpha\in\Sigma(F)\; | \; \alpha(C)>0, \;\; wu(\alpha)(C)< 0 \;\; \text{and} \;\; w(\alpha)\in \Sigma(F')\}.
\end{align*}
Again, the number of elements of $X(u)$ is independent of $u$ and equal to 
\[
\{\alpha\in\Sigma\backslash\Sigma(F)\; | \; \alpha(C)>0 \;\; \text{and} \;\; w(\alpha)(F')< 0\}.
\]
The set $Y(u)$ is contained in 
\[
Y=\{\alpha\in\Sigma\backslash\Sigma(F)\; | \; \alpha(C)>0 \;\; \text{and} \;\; w(\alpha)(F')= 0\}
\]
It is enough to prove that there exists a $u\in\widetilde{W}_{F'}$ with $Y(u)=Y$ and that those $u$ form a single right coset modulo $\widetilde{W}_{F'}\cap(w\widetilde{W}_{F}w^{-1})$. The set $Y(u)$ (resp. $Y$) is the image under the map $\beta\mapsto w^{-1}(\beta)$ of the set
\begin{align*}
& U(u)=\{\beta\in\Sigma(F')\;|\; w^{-1}(\beta)(F)<0 \;\;\text{and}\;\; u(\beta)(C)<0\} \\
& (\text{resp.} \quad U=\{\beta\in\Sigma(F')\;|\; w^{-1}(\beta)(F)<0\}).
\end{align*}
The condition $U(u)=U$ is equivalent to $u(U)\subset\Sigma(F')^-$, where $\Sigma(F')^-$ is the set $\Sigma(F')^-=\{\beta\in\Sigma(F')\;|\; \beta(C)<0\}$. This last set is a Borel subset of $\Sigma(F')$. Set
\[
L=\{\beta\in\Sigma(F')\;|\; w^{-1}(\beta)\in\Sigma(F)\}.
\]
Then
\[
P=L\cup U=\{\beta\in\Sigma(F')\;|\; w^{-1}(\beta)\leq 0\}
\]
is a parabolic subset of $\Sigma(F')$ with Levi subset $L$ and unipotent subset $U$. The inclusion $u(U)\subset\Sigma(F')^-$ is equivalent to $\Sigma(F')^-\subset u(P)$. It is well known that there is a $u\in\widetilde{W}_{F'}$ that realizes this inclusion and that it is unique modulo right multiplication with elements of the Weyl group associated to $L$. But this group is contained in the group $\widetilde{W}_{F'}\cap(w\widetilde{W}_{F}w^{-1})$. This proves the Lemma and also Lemma \ref{representants} (ii).
\end{proof}

\noindent For any $\mu\in X_*(T)_I$ denote by $\mu_{\anti}$ the unique \footnote{The uniqueness follows from the fact that $W_0$ acts trivial on the torsion elements in $X_*(T)_I$.} antidominant representative in $W_0\cdot \mu$.

\begin{cor}\label{specialreps}
There is the identification 
\[
\iweyl{\{x\}}{\{x\}}=\{e^{\mu} \;|\; \mu\in X_*(T)_I \;\;\;\text{and}\;\;\; \mu=\mu_{\anti} \},
\] 
and for $e^{\mu}\in\iweyl{\{x\}}{\{x\}}$, its length is given by
\[
l(e^{\mu})=|\langle\mu,2\rho\rangle|,
\]
where $\rho$ denotes the halfsum of the positive roots in $\Sigma_0$. \\
For $\lambda, \mu$ antidominant, $e^{\lambda}\leq e^{\mu}$ in the Bruhat order if and only if $\lambda\leq\mu$ in the antidominance order with respect to $\Sigma_0$.
\end{cor}

\begin{proof}
Let $\mu\in X_*(T)_I$. Since $\Sigma(\{x\})=\Sigma_0$, we deduce from \eqref{rightmin} that
\begin{equation}\label{rightminspecial}
l((e^\mu)^{\{x\}})=\sum_{a(C)>0}|\langle\mu,a\rangle|\;\; -\;\; \#\{a\in\Sigma_0\;|\;a(C)>0 \;\;\text{and}\;\; \langle\mu,a\rangle> 0\}.
\end{equation}
This is maximal, for $\mu$ varying in its $W_0$-orbit, if $\mu=\mu_{\anti}$ is the unique antidominant element, and in this case
\[
l(e^\mu)=\sum_{a(C)>0}|\langle\mu,a\rangle|=|\langle\mu,2\rho\rangle|.
\] 
The last assertion is well-known.
\end{proof}

\begin{rem}
Formula \eqref{rightminspecial} is also given by Iwahori and Matsumoto in \cite[Prop. 1.25]{IwahoriMatsumoto}.
\end{rem}

\section{Schubert varieties}\label{schubert}
First we recall some facts on Schubert varieties in twisted affine flag varieties as given by G. Pappas and M. Rapoport in \cite{PRtwisted}. We then explain how to calculate the dimension and stratification in the case of a twisted affine Gra§mannian.

 We assume from now on that $K$ is of equal characteristic, i.e. $\chara(K)=\chara(k)$. We give the integers $\mathcal{O}_K$ the structure of a $k$-algebra. 
 
Let $G$ be a scheme over $K$. Then the functor $LG$ from the category of affine $k$-schemes to the category of sets is defined by
\[
LG(S)\define G(\Spec((R\:\widehat{\otimes}_k\,\mathcal{O}_K)\otimes_{\mathcal{O}_K}K)),
\]
for any affine $k$-scheme $S=\Spec(R)$. Here $R\:\widehat{\otimes}_k\,\mathcal{O}_K$ denotes the completed tensor product with respect to the $\mathfrak{m}_K$-adic topology. Likewise, for any scheme $\mathcal{G}$ over $\mathcal{O}_K$ one defines the functor $L^+\mathcal{G}$ as
\[
L^+\mathcal{G}(S)\define\mathcal{G}(\Spec(R\:\widehat{\otimes}_k\,\mathcal{O}_K)),
\]
for any affine $k$-scheme $S=\Spec(R)$. The functors $LG$ and $L^+\mathcal{G}$ give rise to sheaves in the fpqc-topology on the category of affine $k$-schemes. 

\begin{rem}
Choosing a uniformizer $t$ of $\mathcal{O}_K$, we obtain
\[
LG(S)=G(\Spec(R\rpot{t})) \qquad \text{and} \qquad L^+\mathcal{G}(S)=\mathcal{G}(\Spec(R\pot{t})),
\]
for any affine $k$-scheme $S=\Spec(R)$.
\end{rem}

If $G$ (resp. $\mathcal{G}$) is a group scheme over $K$ (resp. $\mathcal{O}_K$), then $LG$ (resp. $L^+\mathcal{G}$) is a sheaf of groups. Let $\mathcal{G}$ be a group scheme over $\mathcal{O}_K$. Then one defines $\mathcal{F}_\mathcal{G}$ as the fpqc-quotient 
\[
\mathcal{F}_\mathcal{G}\define L\mathcal{G}_\eta/L^+\mathcal{G},
\] 
where $\mathcal{G}_\eta$ denotes the generic fiber of $\mathcal{G}$ over $\mathcal{O}_K$. In other words, $\mathcal{F}_\mathcal{G}$ is the fpqc-sheaf on the category of affine $k$-schemes associated to the functor $S \mapsto L\mathcal{G}_\eta(S)/L^+\mathcal{G}(S)$ for any affine $k$-scheme $S$.

\begin{thm}[\cite{PRtwisted}, Thm. 1.4]\label{represent}                                                                   
\noindent\emph{(i)} If $G$ is affine of finite type over $K$, then $LG$ is representable by a strict ind-affine scheme over $k$. \\
\noindent\emph{(ii)} If $\mathcal{G}$ is affine of finite type over $\mathcal{O}_K$, then $L^+\mathcal{G}$ is representable by an affine scheme over $k$. \\
\noindent\emph{(iii)} If $\mathcal{G}$ is a smooth affine group scheme over $\mathcal{O}_K$, then $\mathcal{F}_\mathcal{G}$ is representable by an ind-scheme of ind-finite type over $k$ and the quotient map $L\mathcal{G}_\eta \to \mathcal{F}_\mathcal{G}$ admits sections \'etale-locally.
\end{thm}

Let $G$ be a connected reductive linear algebraic group over $K$.

\begin{definition}
Let $F$ be a facet of the building $\mathscr{B}(G,K)$ and denote by $\mathcal{P}_F$ the corresponding parahoric group scheme. The fpqc-sheaf 
\[
\mathcal{F}_F=\mathcal{F}_{\mathcal{P}_F}=LG/L^+\mathcal{P}_F
\]
is called the \emph{(twisted) affine flag variety}.\\ 
If $F$ is an alcove (all of these are conjugate), we call $\mathcal{F}_{\mathcal{P}_F}$ the \emph{(twisted) full flag variety}. \\
If $F$ is a special vertex, we call $\mathcal{F}_{\mathcal{P}_F}$ the \emph{(twisted) affine Gra§mannian associated to $F$}.
\end{definition}

\begin{rem}
By Pappas and Rapoport \cite[Thm. 5.1]{PRtwisted}, the Kottwitz morphism gives rise to a locally constant surjective morphism of ind-group schemes over $k$ 
\[
\kappa_G: LG \longrightarrow \underline{\pi_1(G)_I},
\]
where $\underline{\pi_1(G)_I}$ denotes the constant group scheme associated to $\pi_1(G)_I$. In particular, the $k$-points of the neutral component  $(LG)^0(k)$ are equal to $G_1$ (cf. \eqref{doublesystem}).
\end{rem}

Fix a parahoric subgroup $P'=P_{F'}$ (resp. $P=P_{F}$) given by a facet $F'$ (resp. $F$) of $\mathscr{B}=\mathscr{B}(G,K)$, and denote by $\mathcal{P}'$ (resp. $\mathcal{P}$) the corresponding parahoric group scheme. We choose the maximal $K$-split torus $S$ such that $F'$ and $F$ are contained in the corresponding apartment $\mathscr{A}=\mathscr{A}(G,S,K)$.

\begin{definition}
Let $w$ be an element of the Iwahori-Weyl group $\widetilde{W}=\widetilde{W}(G,S)$.\\
\noindent \emph{(i)} The \emph{$(P',P)$-Schubert cell} $C_w=C_w(P',P)$ is the reduced subscheme 
\[
L^+\mathcal{P}'\cdot n_w \subset LG/L^+\mathcal{P}=\mathcal{F}_\mathcal{P},
\]
where $n_w$ denotes any representative in $N(K)$.\\
\noindent\emph{(ii)} The \emph{$(P',P)$-Schubert variety} $S_w=S_w(P',P)$ is the reduced scheme with underlying set the Zariski closure of $C_w$.
\end{definition}

\begin{rem}
The definition of $C_w$ and hence of $S_w$ is independent of the chosen representative $n_w$ of $w$. The Schubert cell $C_w$ is an irreducible smooth variety over $k$, since the $L^+\mathcal{P}'$-action on $C_w$ factors through a finite-dimensional irreducible smooth quotient of $L^+\mathcal{P}'$. Note that in general $C_w$ is not a topological cell, when $P'$ is not an Iwahori subgroup. The Schubert variety $S_w$ is an irreducible projective variety, and $C_w$ is an open dense subset in $S_w$. However, $S_w$ is not smooth in general, and there arise interesting singularities along the boundaries of its stratification in $L^+\mathcal{P}'$-orbits.
\end{rem}

\begin{thm}[\cite{PRtwisted}, Thm. 8.4]\label{normality}
Suppose that $G$ splits over a tamely ramified extension of $K$ and that the order of the fundamental group of the derived group $\pi_1(G_{\drv})$ is prime to the characteristic of $k$. Then for any $w\in\widetilde{W}$ the Schubert variety $S_w$ is normal, Frobenius-split (when $\chara(k)>0$) and has rational singularities.
\end{thm} 

Let $C$ be an alcove of $\mathscr{A}$, and assume that $F'$ and $F$ are contained in its closure. Recall that we have a semidirect product decomposition $\widetilde{W}=\pi_1(G)_I\ltimes W_{\aff}$. Denote by $l$ the length function of $\widetilde{W}$ with respect to $C$.

\begin{prop} \label{enumeratecells}
Let $w\in\iweyl{F'}{F}$. \\
\emph{(i)} The Schubert variety $S_w=S_w(P',P)$ is set theoretically the disjoint union of locally closed subvarieties
\[
S_w=\coprod_{\substack{v\,\in\,\iweyl{F'}{F} \\ v \,\leq \, w}} C_v(P',P).
\] 
\emph{(ii)} The dimension of $S_w$ is $l(w)$.
\end{prop}

\begin{proof} 
We may assume after translation with some element of $\pi_1(G)_I$ that $S_w$ is contained in the neutral component $(\mathcal{F}_F)^0$ and so $w\in W_{\aff}$. Let $B=P_C$ be the Iwahori subgroup corresponding to the alcove $C$. Since $(LG)^0(k)=\ker(\kappa_G)=G_1$, the $k$-points $S_w(k)$ coincide with a $(P',P)$-Schubert variety for the double Tits-system $(G_1,B,N_1,\mathbb{S})$ (cf. \eqref{doublesystem}). \\
First assume that $P=B$ is the Iwahori subgroup. Denote the corresponding Iwahori group scheme by $\mathcal{B}$. Using the Demazure resolution in \cite[Prop. 9.6 (b)]{PRtwisted}, we see that the closure of an $L^+\mathcal{B}$-orbit is the disjoint union of the $L^+\mathcal{B}$-orbits for elements which are smaller than $w$ with respect to the Bruhat order. By Lemma \ref{doubleclasses} and Equation \eqref{representatives}, this implies (i). Using the Demazure resolution again, we see that the dimension of an $L^+\mathcal{B}$-orbit is exactly the length of the corresponding element in $W_{\aff}$. Hence, the $L^+\mathcal{B}$-orbit of $w$ is the unique $L^+\mathcal{B}$-orbit of maximal dimension in $S_w$. So we have $\dim(S_w)=l(w)$. This proves (ii). \\
If $P$ is not neccessarily an Iwahori subgroup, we consider the projection $p:\mathcal{F}_C\to\mathcal{F}_F$ from the full flag variety to the partial flag variety. Then $p^{-1}(S_w)$ is a $(P',B)$-Schubert variety and its projection onto $\mathcal{F}_F$ is $S_w$. This implies (i). Since $C_w^B=C_w(B,P)$ is open in $S_w$, it is enough to determine $\dim(C_w^B)$. We have a smooth, surjective morphism 
\[
p\vert_{p^{-1}(C_w^B)}: p^{-1}(C_w^B)\longrightarrow C_w^B
\]
and so $\dim(C_w^B)=\dim(p^{-1}(C_w^B))-\dim(p^{-1}(w))$. We show that there is set theoretically a disjoint union
\begin{equation}
p^{-1}(C_w^B)=\coprod_{u\in\widetilde{W}_F} C_{wu}(B,B). \label{claim}
\end{equation}
But $P=\coprod_{u\in\widetilde{W}_F}BuB$ by general properties of Tits-systems, and hence 
\[
BwP=\coprod_{u\in\widetilde{W}_F}BwBuB.
\]
Since $w=w^F$, we have $l(wu)=l(w)+l(u)$ for all $u\in\widetilde{W}_F$, and therefore $BwBuB=BwuB$. This proves \eqref{claim}. \\
By the Iwahori case we have
\[
\dim(p^{-1}(C_w^B))=\underset{u\in\widetilde{W}_F}{\Max}l(wu).
\]
But $l(wu)=l(w)+l(u)$ and the right hand side is equal to $l(w)+l(w_0)$, where $w_0$ is the longest element in $\widetilde{W}_F$. Using Remark \ref{reductivequotient} (i) below we get 
\[
\dim(p^{-1}(w))=\dim(L^+\mathcal{P}/L^+\mathcal{B})=l(w_0),
\]
and we are done.
\end{proof}

\begin{rem} \label{reductivequotient}
(i) Let $\bar{P}_F^{\red}$ be the maximal reductive quotient of the special fiber of $\mathcal{P}_F$. The maximal $K$-split torus $S$ has a natural $\mathcal{O}_K$-structure, and its special fiber $\bar{S}$ is a maximal torus of $\bar{P}_F^{\red}$ (cf. \cite[3.5]{Ti}). The group $\widetilde{W}_F$ can be identified with the relative Weyl group of $\bar{P}_F^{\red}$ with respect to the maximal torus $\bar{S}$ (cf. \cite[3.5.1]{Ti}). Let $f$ be a facet whose closure contains $F$. Consider the special fiber of the induced morphism $\mathcal{P}_f\to\mathcal{P}_F$. The image in $\bar{P}_F^{\red}$ is a parabolic subgroup $\bar{P}_f'$ and reduction mod $\mathfrak{m}_K$ gives an isomorphism
\[
L^+\mathcal{P}_F/L^+\mathcal{P}_f \overset{\cong}{\longrightarrow} \bar{P}_F^{\red}/\bar{P}_f'.
\]
Hence, $L^+\mathcal{P}_F/L^+\mathcal{P}_f$ is a homogeneous space and its dimension is given by the length of the longest element in $(\widetilde{W}_F)^f$, where $(\widetilde{W}_F)^f$ are the minimal length representatives of $\widetilde{W}_F/\widetilde{W}_f$ in $\widetilde{W}_F$.\\
(ii) Lemma \ref{doubleclasses} and some geometric arguments similar to those in the proof of Proposition \ref{enumeratecells} may be used to obtain another proof of Lemma \ref{representants}.
\end{rem}

If $F'=F=\{x\}$ is a special vertex, then $\iweyl{\{x\}}{\{x\}}=X_*(T)_I^-$ by Corollary \ref{specialreps} are the antidominant coweights. Then Proposition \ref{enumeratecells} implies:

\begin{cor}\label{specialschubert}
Let $\mu\in X_*(T)_I^-$ be a antidominant coweight. \\
\emph{(i)} The $(P_{\{x\}},P_{\{x\}})$-Schubert variety $S_\mu$ is set theoretically the disjoint union of locally closed subvarieties
\[
S_\mu=\coprod_{\substack{\lambda\,\in\,X_*(T)_I^- \\ \lambda \,\leq \, \mu}} C_\lambda(P_{\{x\}},P_{\{x\}}),
\] 
where $\lambda\leq\mu$ in the antidominance order.\smallskip \\
\emph{(ii)} The dimension of $S_\mu$ is $|\langle\mu,2\rho\rangle|$.
\end{cor}

\section{Equivariant resolution of Schubert varieties}
We explain the construction of an equivariant affine Demazure resolution in twisted affine flag varieties. This is analogous to the construction in finite dimensional flag varieties explained by S. Billey and V. Lakshmibai in \cite[Ch. 9.1]{BLak}. We thank N. Perrin for explaining a proof to us and his permission to reproduce it. 

Fix an alcove $C$ in the standard apartment $\mathscr{A}=\mathscr{A}(G,S,K)$ of the connected reductive group $G$ over $K$ with fixed maximal $K$-split torus $S$. Let $\mathcal{B}=\mathcal{P}_C$ be the corresponding parahoric group scheme and $B=\mathcal{B}(\mathcal{O}_K)$ the parahoric subgroup. We refer to parahoric subgroups of $G(K)$ containing $B$ as \emph{standard} parahoric subgroups.\\
\noindent Let $\widetilde{W}=\pi_1(G)_I\ltimes W_{\aff}$ be the Iwahori-Weyl group and denote by $\mathbb{S}$ the set of simple reflections at the walls of $C$. Recall (cf. \eqref{doublesystem}) that the quadruple $(G_1,B,N_1,\mathbb{S})$ is a double Tits-system with affine Weyl group $W_{\aff}$. For any facet $F$ contained in the closure of $C$, the group $\widetilde{W}_F$ is a parabolic subgroup of the Coxeter-system $(W_{\aff},\mathbb{S})$ which is generated by the simple reflections $\mathbb{S}_F$ at the walls of $C$ containing $F$. Let $\widetilde{W}^F$ be the set of minimal length representatives of $\widetilde{W}/\widetilde{W}_F$, i.e.
\[
\widetilde{W}^F\define\{w\in\widetilde{W}\;|\;w=w^F\}=\{w\in\widetilde{W}\;|\; l(ws)=l(w)+1\;\;\forall\;s\in\mathbb{S}_F\}.
\]
Analogously, we define ${^F\hspace{-0.01cm}\widetilde{W}}$ as the set of minimal length representatives of $\widetilde{W}_F\backslash\widetilde{W}$, i.e. 
\[
{^F\hspace{-0.01cm}\widetilde{W}}\define\{w\in\widetilde{W}\;|\;w^{-1}\in\widetilde{W}^F\}=\{w\in\widetilde{W}\;|\; l(sw)=l(w)+1\;\;\forall\;s\in\mathbb{S}_F\}. 
\]
For any $w\in\widetilde{W}$, we may write
\[
w=w^F\cdot w_F \quad \text{(resp.} \;\, w={_Fw}\cdot {^F\!w} \,\text{)},
\]
 where $w^F\in \widetilde{W}^F$ (resp. ${^F\!w}\in{^F\widetilde{W}}$) and so $w_F\in\widetilde{W}_F$ (resp. ${_Fw}\in\widetilde{W}_F$). We sometimes replace the subscript (resp. superscript) $F$ by the corresponding parahoric subgroup.

Fix a facet $F$ in the closure of $C$ and denote by $P_F$ (resp. $\mathcal{P}_F$) the corresponding standard parahoric subgroup (resp. parahoric group scheme). Let $w\in W_{\aff}$ and $S_w$ the $(B,P_F)$-Schubert variety in $\mathcal{F}_F$ corresponding to $w$. 

\begin{definition} 
A sequence of parahoric group schemes $\mathcal{P}_1$,...,$\mathcal{P}_n$ and $\mathcal{Q}_1$,...,$\mathcal{Q}_{n}$ with 
\begin{align*}
& \rm(i) && L^+\!\mathcal{Q}_i\subset L^+\mathcal{P}_i\cap L^+\mathcal{P}_{i+1} \quad \text{for}\;\; i=1,...,n-1, \\
& \rm(ii) && L^+\!\mathcal{Q}_n\subset L^+\mathcal{P}_n \cap L^+\mathcal{P}_F
\end{align*}
is called \emph{resolutive with respect to} $S_w$, if the morphism given by multiplication
\[
L^+\mathcal{P}_1 \times^{L^+\!\mathcal{Q}_1}...\times^{L^+\!\mathcal{Q}_{n-1}} L^+\mathcal{P}_n/L^+\!\mathcal{Q}_n \longrightarrow \mathcal{F}_F
\]
factors through $S_w$ and induces a birational morphism
\[
m:L^+\mathcal{P}_1 \times^{L^+\!\mathcal{Q}_1}...\times^{L^+\!\mathcal{Q}_{n-1}} L^+\mathcal{P}_n/L^+\!\mathcal{Q}_n \longrightarrow S_w.
\] 
\end{definition}

The morphism $m$ is $L^+\mathcal{P}_1$-equivariant, and hence $L^+\mathcal{P}_1$ must stabilize the Schubert variety $S_w$. Moreover, by Remark \ref{reductivequotient} the source of $m$ is an iterated extension of homogeneous spaces and hence smooth.

\begin{examples}
(i) Let $w\in (W_{\aff})^F$ and fix a reduced decomposition $(s_1,...,s_n)$ of $w$. Denote by $\mathcal{P}_i$ the parahoric group scheme corresponding $s_i$. Set $\mathcal{Q}_i=\mathcal{B}$ for $i=1,...,n$. Then the sequence $\mathcal{P}_1,...,\mathcal{P}_n$ and $\mathcal{Q}_1,...,\mathcal{Q}_n$ is resolutive for $S_w$.\smallskip \\
(ii) Let $G$ be split and simple. Assume that $F=\{x\}$ is a hyperspecial vertex and that $w=e^{-\alpha^{\vee}}$, where $-\alpha^{\vee}\in X_*(T)^-$ is the unique short antidominant\footnote{The sign comes from the normalization of the Kottwitz morphism in Remark \ref{normalization}.} coroot with respect to the maximal torus $T$. Let $\mathcal{P}_{\alpha^{\vee}}$ be the unique standard parahoric group scheme corresponding to the simple (finite) reflections $\mathbb{S}_{\alpha^{\vee}}$ which stabilize $\alpha^{\vee}$. The simple affine reflection is $s_0=e^{\alpha^{\vee}}s_{\alpha}$, where $s_{\alpha}$ is the reflection about the hyperplane $\ker(\alpha)\subset X_*(T)$. Denote by $\mathcal{P}_{s_0,\alpha^{\vee}}$ the unique standard parahoric group scheme corresponding to the set of simple reflections $\{s_0\}\cup\mathbb{S}_{\alpha^{\vee}}$. Set 
\[
\mathcal{P}_1=\mathcal{P}_{\{x\}} \quad\text{and}\quad \mathcal{P}_2=\mathcal{P}_{s_0,\alpha^{\vee}}\\
\]
and
\[
\mathcal{Q}_1=\mathcal{P}_{\alpha^{\vee}} \quad\text{and}\quad \mathcal{Q}_2=\mathcal{P}_{\alpha^{\vee}}.\\
\]
Then the sequence $\mathcal{P}_1,\mathcal{P}_2$ and $\mathcal{Q}_1,\mathcal{Q}_2$ is resolutive for $S_w$ in the affine Gra§mannian $\mathcal{F}_{\{x\}}$. This is the resolution given by Ng\^o and Polo in \cite[Lemme 7.3]{NgoPolo}. 
\end{examples}

\begin{lem}
There is a unique parahoric group scheme $\mathcal{P}_w$ such that $L^+\mathcal{P}_w$ stabilizes $S_w$, and such that $P_w=L^+\mathcal{P}_w(k)$ is maximal among all standard parahoric subgroups with this property. Furthermore,
\[
\widetilde{W}_{P_w}=\spanned\{s\in\mathbb{S}\;|\; l((sw)^F) \leq l(w^F)\}\subset W_{\aff}.
\]
\end{lem}
\begin{proof}
For any $s\in\mathbb{S}$, the properties $l((sw)^F)\leq l(w^F)$ and $(sw)^F\leq w^F$ are equivalent. Let $P_w$ be the standard parahoric subgroup corresponding to the set of simple reflections $\{s\in\mathbb{S}\;|\; (sw)^F \leq w^F\}$. By Proposition \ref{enumeratecells} (i), the group scheme $L^+\mathcal{P}_w$ stabilizes $S_w$ and $P_w$ is maximal with this property.
\end{proof}

The group scheme $L^+\mathcal{P}_w$ is called the \emph{stabilizer of} $S_w$.

\begin{thm}\label{resolsimplycon}
Let $w\in (W_{\aff})^F$ and let $S_w$ be the corresponding $(B,P_F)$-Schubert variety. Then there exists a unique decomposition $w={w_1}\cdot...\cdot{w_n}$ with $$l(w)=l({w_1})+...+l({w_n}),$$ and with the following property: define the sequence of parahoric group schemes $\mathcal{P}_1$,...,$\mathcal{P}_n$ and $\mathcal{Q}_1$,...,$\mathcal{Q}_{n}$ by the following conditions,
\begin{align*}
& \emph{(i)} && L^+\mathcal{P}_i \quad \text{is the stabilizer of the} \;\;(B,P_F) \text{-Schubert variety given by} \quad {w_i}\cdot...\cdot{w_n} \\ & &&\text{for}\quad i=1,...,n, \\
& \emph{(ii)} && L^+\!\mathcal{Q}_i=L^+\mathcal{P}_i\cap L^+\mathcal{P}_{i+1} \quad \text{for}\quad i=1,...,n-1, \\
& \emph{(iii)} && L^+\!\mathcal{Q}_n=L^+\mathcal{P}_n\cap L^+\mathcal{P}_F.
\end{align*}
Then $\mathcal{P}_1$,...,$\mathcal{P}_n$ and $\mathcal{Q}_1$,...,$\mathcal{Q}_{n}$ is resolutive for $S_w$.  
\end{thm}

\begin{cor}
Let $F'$ be another facet in the Bruhat-Tits building of $G$ and let $S(F',F)$ be any $(F',F)$-Schubert variety in $\mathcal{F}_F$. Then there exists a resolutive sequence of parahoric group schemes for $S(F',F)$ such that the corresponding birational morphism is $L^+\mathcal{P}_{F'}$-equivariant. 
\end{cor}  
\vspace{-0,95cm}
\hfill \ensuremath{\Box}

\begin{proof}[Proof of Theorem \ref{resolsimplycon} (N. Perrin).]
Let $L^+\mathcal{P}=L^+\mathcal{P}_1$ be the stabilizer of $S_w$, and denote by $P=L^+\mathcal{P}(k)$ the corresponding parahoric subgroup. Write $w={_Pw}\cdot{^P\! w}$ and define $w_1={_Pw}$. Note that ${_Pw}$ is non-trivial, if $w$ is non-trivial. So $L^+\mathcal{Q}=L^+\mathcal{Q}_1$ is the intersection of $L^+\mathcal{P}$ with the stabilizer of the $(B,P_F)$-Schubert variety $S_v$ corresponding to $v={^P\!w}$. \smallskip \\
\textbf{claim:} \emph{The element $w_1={_Pw}$ is the longest element in $(\widetilde{W}_P)^{Q}$.} \smallskip \\
Let $s\in\mathbb{S}_P$, and assume that $l(sw_1)=l(w_1)+1$. Since $sw_1\in\widetilde{W}_P$, we get that $l(sw_1v)=l(sw_1)+l(v)$, which means $l(sw)=l(w)+1$. Hence, $sw\geq w$ in the Bruhat order. But $s\in\widetilde{W}_P$ implies $(sw)^F=w^F=w$, i.e. $w^{-1}sw\in\widetilde{W}_F$. We obtain
\[
v=v^F=(vw^{-1}sw)^F=(w_1^{-1}sw_1v)^F,
\]
and so $w_1^{-1}sw_1\in\widetilde{W}_Q$. This implies that $l((sw_1)^Q)\leq l(w_1^Q)$, and hence the claim. \smallskip \\
We next show that the morphism 
\begin{equation}\label{indstep}
L^+\mathcal{P}\times^{L^+\!\mathcal{Q}}S_v\longrightarrow S_w
\end{equation}
given by multiplication is birational.\\ 
In the following, we equip all orbits with their reduced scheme structure. By the claim above, the orbit $(L^+\!\mathcal{B}\,{_Pw}\,L^+\!\mathcal{Q})/L^+\!\mathcal{Q}$ is the open $L^+\!\mathcal{B}$-orbit in $L^+\!\mathcal{P}/L^+\!\mathcal{Q}$. We obtain that $(L^+\!\mathcal{B}\,{_Pw}\,L^+\!\mathcal{Q})\times^{L^+\!\mathcal{Q}}(L^+\!\mathcal{Q}\,{^P\!w}\,L^+\!\mathcal{P}_F)/L^+\!\mathcal{P}_F$ is open in $L^+\!\mathcal{P}\times^{L^+\!\mathcal{Q}}S_v$ and, via the multiplication morphism, isomorphic to $(L^+\!\mathcal{B}\,w\,L^+\!\mathcal{P}_F)/L^+\!\mathcal{P}_F$. This proves that the morphism \eqref{indstep} is birational. The theorem now follows by induction on $n$.
\end{proof}

\section{Example of a ramified unitary group in odd dimension}\label{unitary}
We make Theorem \ref{resolsimplycon} explicit in the case of a twisted affine Gra§mannian for a ramified quasi-split unitary group. In the presentation of the material we follow \cite{PRunitary}.

Assume that $\chara(k)\neq 2$. Fix a quadratic totally ramified extension $\tilde{K}$ of $K$ and a uniformizer $u\in\mathcal{O}_{\tilde{K}}$ with $u^2=t$, where $t$ is a uniformizer of $\mathcal{O}_K$. We extend the valuation of $K$ to $\tilde{K}$, i.e. $\val(u)={1\over 2}$. Denote by $\overline{\cdot}\in I=\Gal(\tilde{K}/K)$ the non-trivial element of the Galois group. Let $G=\SU(W,\phi)$ be the special unitary group for a hermitian vector space $(W,\phi)$ of odd dimension $n=2m + 1\geq 3$ over $\tilde{K}$, i.e.
\[
G(R)=\{g\in \SL(W\otimes_KR)\;|\; \phi(gv,gw)=\phi(v,w) \;\;\forall\; v,w\in W\otimes_KR\}
\]
for any $K$-algebra $R$. Assume that $\phi$ splits, i.e. that there exists a basis $e_1,...,e_n$ of $W$ such that
\[ 
\phi(e_i,e_{n-j+1})=\delta_{i,j} \qquad \forall \;\; i,j=1,...,n.
\] 
For $0\leq i\leq n-1$, set
\[
\Lambda_i=\spanned_{\mathcal{O}_{\tilde{K}}}\{u^{-1}e_1,...,u^{-1}e_i,e_{i+1},...,e_n\},
\]
and complete $\{\Lambda_0,\ldots,\Lambda_{n-1}\}$ into a selfdual periodic lattice chain $(\Lambda_i)_{i\in\mathbb{Z}}$ by defining $\Lambda_{qn+i}=u^{-q}\Lambda_i$ for $i\in\{0,\ldots,n-1\}$. The dual $\Lambda_i^{\vee}$ with respect to $\phi$ of the lattice $\Lambda_i$ is given by $\Lambda_i^{\vee}=\Lambda_{-i}$. \\
For any non-empty subset $J\subset\{0,\ldots,m\}$ we may consider the partial periodic lattice chain $\Lambda_\bullet^J=(\Lambda_i)_i$ of type $J$, for $i \in J\cup(-J)$ modulo $n$. Every parahoric subgroup of $G(K)$ is conjugate to the stabilizer of a partial selfdual periodic lattice chain of some uniquely defined type $J$ (cf. \cite[\S 4]{PRtwisted}). In this way we get a bijection between non-empty subsets of $\{0,...,m\}$ and conjugacy classes of parahoric subgroups of $G(K)$. The subsets $J=\{0\}$ and $J=\{m\}$ correspond to the special parahoric subgroups. Let $\mathcal{P}_J$ be the parahoric group scheme which corresponds to the stabilizer of the lattice chain $\Lambda_\bullet^J$. \\
Now let $\mathcal{P}=\mathcal{P}_J$ with $J=\{m\}$. Let $\mathcal{F}=LG/L^+\mathcal{P}$ be the corresponding twisted affine Gra§mannian. By \cite[Thm. 4.1]{PRtwisted}, there is a functorial bijection of sets between $\mathcal{F}(\Spec(R))$ and
\[
\left\{\left.
\begin{aligned}
& \text{partial selfdual periodic} \\ 
& R\pot{u}\text{-lattice chains}\;\;(\mathcal{L}_i)_i \\
& \text{for}\;\; i\in\{m,-m\} \bmod n \quad
\end{aligned}
 \right| 
\begin{aligned}
\text{(i)} & \quad \wedge^n\mathcal{L}_m = (u^{-m}) \\
 \text{(ii)} & \quad \mathcal{L}_{m+1}/\mathcal{L}_m \;\; \text{is locally free on} \;\; R \;\; \text{of rank} \;\; 1   
\end{aligned}
\right\}.
\]
Let $S$ be the standard (diagonal) maximal split torus of $G$ whose $K$-valued points are 
\[
\{\text{diag}(a_1,\ldots,a_m,1,a_m^{-1},\ldots,a_1^{-1})\;|\;a_i\in K \;\;\forall\;i=1,\ldots,m\}.
\]
Its centralizer $T$ is the maximal torus whose $K$-valued points are
\[
\{\text{diag}(a_1,\ldots,a_m\,,\,a^{-1}\overline{a}\,,\,\overline{a}^{-1}_m,\ldots,\overline{a}^{-1}_1)\;|\;a=a_1\cdot\ldots\cdot a_m\;\;\;\text{and}\;\;\;a_i\in \tilde{K} \;\;\forall\;i=1,\ldots,m\}.
\]
We fix the Borel subgroup of upper triangular matrices in $G$, and the alcove contained in the corresponding positive Weyl chamber. \\
The group $G$ is simply connected and hence $\widetilde{W}=W_{\aff}$ by Remark \ref{simplyrem}. For $p=0,\ldots,m$, define
\begin{equation}\label{ele}
e^{-\mu_p}\define [\text{diag}(u^{(p)},1,\ldots,1,(-1)^p,1,\ldots,1,(-{u^{-1}})^{(p)})]\in W_{\aff},
\end{equation}
where $(-1)^p$ is the $m+1$-th entry of the diagonal matrix.
We are interested in determining the resolution provided by Theorem \ref{resolsimplycon} of the $(\mathcal{P},\mathcal{P})$-Schubert varieties $S_p\subset \mathcal{F}$ corresponding to $e^{-\mu_p}$ in terms of lattice chains. \\
First we make the Kottwitz morphism $\kappa_T:T(K)\rightarrow X_*(T)_I$ and the identification of the affine Weyl group with the affine Weyl group of a reduced root system explicit. \\
We identify 
\[
X_*(T)=\{(x_1,\ldots,x_n)\in\mathbb{Z}^n\;|\;\sum{x_i}=0\}\subset \mathbb{Z}^n
\]
and $X^*(T)=\mathbb{Z}^n/\mathbb{Z}$ with $\mathbb{Z}\hookrightarrow\mathbb{Z}^n$ embedded diagonally. The non-trivial element of the Galois group $I$ acts on $\mathbb{Z}^n$ by
\[
(x_1,\ldots,x_n)\mapsto (-x_n,\ldots,-x_1).
\] 
The pairing $X_*(T)\times X^*(T)\rightarrow\mathbb{Z}$ is induced by the standard pairing $\mathbb{Z}^n\times\mathbb{Z}^n\rightarrow\mathbb{Z}$. This gives rise to a commutative diagram
\begin{equation}\label{stdpairing}
\begin{CD}
X_*(T)_I\times X^*(T)^I @>>> \mathbb{Z}  \\
                   @|                                  @| \\
               \mathbb{Z}^m\times\mathbb{Z}^m @>{\phantom{\text{long label}}}>> \mathbb{Z}.
\end{CD}
\end{equation}
Hence, the pairing $\langle\str,\str\rangle:\mathbb{Z}^m\times\mathbb{Z}^m\longrightarrow \mathbb{Z}$ in \eqref{stdpairing} is the standard pairing. \\ 
Set $V=X_*(T)_I\otimes\mathbb{R}=\mathbb{R}^m$. Since $X_*(T)_I$ is torsion free, the Kottwitz morphism is uniquely determined by 
\[
\langle\kappa_T(a),e_i\rangle=-\val(a_i\overline{a}_i),
\]
for $a\in T(K)$ and $e_i$ the $i$-th standard basis vector of $V^{\vee}=\mathbb{R}^m$, i.e. we have \footnote{The sign in \eqref{kottunit} comes from our normalization of the Kottwitz morphism in Remark \ref{normalization}.}
\begin{equation}\label{kottunit}
\begin{split}
\kappa_T: T(K) & \longrightarrow X_*(T)_I=\mathbb{Z}^m  \\
\text{diag}(a_1,\ldots,a_n) & \longmapsto (-2\val(a_1),\ldots,-2\val(a_m)). 
\end{split}
\end{equation}
We identify $X_*(S)$ via the canonical mapping $X_*(S)\subset X_*(T)\rightarrow X_*(T)_I=\mathbb{Z}^m$ with $2\mathbb{Z}^m$. By \cite[2.4.2]{PRunitary} the affine root system of $G$ with respect to $S$ (in the chosen basis of $V^{\vee}$) is given by
\[
\pm {1\over 2}e_i\pm {1\over 2}e_j + {1\over 2}\mathbb{Z}, \qquad \pm {1\over 2}e_i + {1\over 2}\mathbb{Z} \quad \text{and} \quad \pm e_i + {1\over 2} + \mathbb{Z}.
\]
Hence, the set of root hyperplanes is the zero set of the affine functions 
\[
\{\pm e_i\pm e_j + \mathbb{Z}, \pm 2e_i + \mathbb{Z}\},
\]
i.e. the corresponding reduced root system $\Sigma_0$ is of type $C_m$. By the choice of the base lattice $\Lambda_m$ the affine Weyl group $W_{\aff}$ is represented as a semidirect product
\begin{equation}\label{semiunit}
W_{\aff}=Q^{\vee}(\Sigma_0)\rtimes W(\Sigma_0)=\mathbb{Z}^{m}\rtimes W_0 \hspace{0.5cm} \text{with} \hspace{0.5cm} W_0=\mathfrak{S}_m\rtimes \{\pm 1\}^{m}.
\end{equation}
The element $e^{-\mu_p}\in W_{\aff}$ defined in \eqref{ele} corresponds via the Kottwitz morphism $\kappa_T$ to the antidominant cocharacter $-\mu_p=(-1^{(p)},0^{(m-p)})$. Denote by $\tau_{i,j}$ the transposition in $\mathfrak{S}_m$ permuting $i$ and $j$ and by $\pm1_i$ the $i$-th entry of $\{\pm1\}^{m}$. Then the simple reflections $s_1,...,s_m$ of $W_0$ are given by $s_i=\tau_{i,i+1}$ for $i=1,...,m-1$ and $s_m=-1_m$. The simple affine reflection $s_0$ is given by $s_0=((1^{(1)},0^{(m-1)}),-1_1)$ with respect to the semidirect product decomposition \eqref{semiunit}. 

\begin{lem}\label{exampledecom}
The decomposition of $e^{-\mu_p}$ (provided by Theorem \ref{resolsimplycon}) is determined by $e^{-\mu_p}=w_{p,1}\cdot w_{p,2}$ with
\[
w_{p,2}=\prod_{i=1}^p (s_0\cdot\ldots\cdot s_{i-1}).
\]
The element $w_{p,2}$ has length $p(p+1)\over 2$ and maps under the quotient mapping $W_{\aff}\rightarrow W_{\aff}/W_0$ to $\mu_p$. 
\end{lem}
\vspace{-0,95cm}
\hfill \ensuremath{\Box}

Denote by $\mathcal{Q}_p$ the standard parahoric group scheme corresponding to the subset of simple reflections $\{s_1,\ldots, \widehat{s}_p,\ldots, s_m\}$. This set generates the stabilizer of $\mu_p$ in $W_0$. Let $\rho=(m,m-1,\ldots,1)\in V^{\vee}$ be the halfsum of the positive roots in $\Sigma_0$.

\begin{cor}\label{exfundcounitary}
The Schubert variety $S_p$ has dimension $|\langle\mu_p,2\rho\rangle|=p(2m+1-p)$, and there is set theoretically a disjoint union in $L^+\mathcal{P}$-orbits given by
\[
S_p=\coprod_{i=0}^pL^+\mathcal{P}e^{-\mu_i}L^+\mathcal{P}/L^+\mathcal{P}.
\]
Moreover, the multiplication morphism 
\[
m: \widetilde{S}_p=L^+\mathcal{P}\times^{L^+\mathcal{Q}_p}(\overline{L^+\mathcal{Q}_p e^{\mu_p}L^+\mathcal{P}})/L^+\mathcal{P}\longrightarrow (\overline{L^+\mathcal{P} e^{-\mu_p}L^+\mathcal{P}})/L^+\mathcal{P}=S_p
\]
is a resolution in the sense of Theorem \ref{resolsimplycon}.
\end{cor}

\begin{proof}
Corollary \ref{specialschubert} implies the statement on the dimension and stratification. The second statement follows from Lemma \ref{exampledecom}  and Theorem \ref{resolsimplycon}.
\end{proof}

\begin{rem}\label{redquotunit}
The maximal reductive quotient $\overline{\mathcal{P}}^{\red}$ of the special fiber $\overline{\mathcal{P}}=\mathcal{P}\otimes_{\mathcal{O}_K}k$ is isomorphic to the symplectic group of a $2m$-dimensional $k$-vector space as follows:  Every element of $\overline{\mathcal{P}}$ commutes with the linear map $u\otimes 1$ on $\Lambda_m\otimes k$, and we obtain a morphism of $k$-groups
\begin{equation}\label{primmorph}
\overline{\mathcal{P}}\longrightarrow \SL(\Lambda_m/u\Lambda_m),
\end{equation}
whose kernel is unipotent. Let $E=\Lambda_m/ \Lambda_m^{\vee}$, a $2m$-dimensional quotient space of $\Lambda_m/u\Lambda_m$. The reduction modulo $u$ of $u\phi$ induces a symplectic form $\phi^{\sympl}$ on $E$, and \eqref{primmorph} induces an isomorphism of $k$-groups
\begin{equation}
\overline{\mathcal{P}}^{\red}\overset{\cong}{\longrightarrow}\SP(E).\label{sympliso}
\end{equation}
\end{rem}

We will now give a linear algebra description of the resolution $m: \widetilde{S}_p\rightarrow S_p$.\\ Let $\mathcal{S}_p$ be the functor on the category of $k$-schemes whose $\Spec(R)$-valued points are the set of $R\pot{u}$-lattices $\Lambda$ such that
\begin{align*}
 \text{(i)} & \qquad \Lambda\subset u^{-1}\Lambda^{\vee} \;\; \text{and}\;\; u^{-1}\Lambda^{\vee}/\Lambda \;\; \text{is locally free on} \;\; R \;\; \text{of rank} \;\; 1 \qquad\;\; \\
\text{(ii)} & \qquad \wedge^n\Lambda=(u^{-m}) \\
\text{(iii)} & \qquad \inv(\Lambda_m\otimes_k R,\Lambda)\leq -\mu_p.
\end{align*}
Let $\widetilde{\mathcal{S}}_p$ be the functor on the category of $k$-schemes whose $\Spec(R)$-valued points are the set of pairs of $R\pot{u}$-lattices $(\Lambda',\Lambda)$ such that
\begin{align*}
 \text{(i)} & \qquad \Lambda \;\; \text{satisfies properties (i)-(iii) above} \\
\text{(ii)} & \qquad (\Lambda')^{\vee}\subset\Lambda' \;\; \text{and}\;\; \Lambda'/(\Lambda')^{\vee} \;\; \text{is locally free on} \;\; R \;\; \text{of rank} \;\; 2(m-p) \\
\text{(iii)} & \qquad \Lambda'\subset\Lambda \;\; \text{and}\;\; \Lambda/\Lambda' \;\; \text{is locally free on} \;\; R \;\; \text{of rank} \;\; p \\
\text{(iv)} & \qquad \Lambda'\subset\Lambda_m \;\; \text{and}\;\; \Lambda_m/\Lambda' \;\; \text{is locally free on} \;\; R \;\; \text{of rank} \;\; p.
\end{align*}
Conditions (ii) and (iv) say that $\Lambda'$ defines a point of $L^+\mathcal{P}/L^+\mathcal{Q}_p$.\\
The functors $\mathcal{S}_p$ and $\widetilde{\mathcal{S}}_p$ are representable by projective schemes over $k$, and there is a natural projection $m':\widetilde{\mathcal{S}}_p\rightarrow \mathcal{S}_p$. \\
Set $\Lambda_m'=\Lambda_m\cap (e^{-\mu_p}\Lambda_m)$. We have morphisms $\varphi: S_p\rightarrow\mathcal{S}_p$ and $\psi: \widetilde{S}_p\rightarrow\widetilde{\mathcal{S}}_p$ given on $\Spec(R)$-valued points by
\begin{align*}
\varphi: S_p(\Spec(R))\longrightarrow\mathcal{S}_p(\Spec(R))\\
g\cdot\mathcal{P}(R\pot{t}) \mapsto g\cdot(\Lambda_m\otimes_kR)
\end{align*}
and
\begin{align*}
\qquad\qquad\qquad\psi: \widetilde{S}_p(\Spec(R))\longrightarrow\widetilde{\mathcal{S}}_p(\Spec(R))\qquad\qquad\qquad\\
(c,h)\cdot(\mathcal{Q}_p(R\pot{t})\times\mathcal{P}(R\pot{t})) \mapsto (c\cdot(\Lambda_m'\otimes_kR),ch\cdot(\Lambda_m\otimes_kR)).
\end{align*}

\begin{prop}\label{demresolunit}
There is a commutative diagram
\[
\begin{CD} 
\widetilde{S}_p  @>{m}>> S_p \\ 
@V{\psi}V{\cong}V  @V{\varphi}V{\cong}V\\ 
 \widetilde{\mathcal{S}}_p @>{m'}>>  \mathcal{S}_p
\end{CD}
\]
with the vertical maps $\varphi$ and $\psi$ being isomorphisms. 
\end{prop}
\begin{proof}
The diagram commutes, and the map $\varphi$ evidently is an isomorphism. An easy calculation shows that $\psi$ is a monomorphism. Since $\psi$ is also projective, it is by [EGA4, Prop. 8.11.5] a closed immersion. We show that $\widetilde{S}_p$ and $\widetilde{\mathcal{S}}_p$ are irreducible and smooth of the same dimension, which will finish the proof.\\ By Corollary \ref{exfundcounitary} the scheme $\widetilde{S}_p$ is irreducible and smooth of dimension $p(2m+1-p)$. Consider the commutative diagram 
\[
\begin{tikzpicture} 
\matrix(a)[matrix of math nodes, 
row sep=3em, column sep=2.5em, 
text height=1.5ex, text depth=0.45ex] 
{\widetilde{S}_p & \widetilde{\mathcal{S}}_p\\ 
L^+\mathcal{P}/L^+\mathcal{Q}_p & \mathscr{G}_{2m-p}(E,\phi^{\sympl}),\\}; 
\path[->](a-1-1) edge node[above] {$\psi$} (a-1-2); 
\path[->](a-1-1) edge node[right] {$\text{pr}_1$} (a-2-1); 
\path[->](a-1-2) edge node[right] {$\text{pr}_2$} (a-2-2);
\path[->](a-2-1) edge node[above] {$j$} (a-2-2); 
\end{tikzpicture}
\]
where $\mathscr{G}_{2m-p}(E,\phi^{\sympl})$ is the finite-dimensional Gra§mannian of $(2m-p)$-dimensional isotropic subspaces $E'$ of the symplectic space $(E,\phi^{\sympl})$, i.e.  $(E')^{\perp}\subset E'$ with respect to the symplectic form $\phi^{\sympl}$ (cf. Remark \ref{redquotunit}). The vertical arrows are given by $\text{pr}_1: (c,h)\mapsto c$ and $\text{pr}_2: (\Lambda',\Lambda)\mapsto \Lambda'/\Lambda_m^{\vee}$. The map $j$ is given by $c\mapsto c\cdot\Lambda_m'/\Lambda_m^{\vee}$. The scheme $L^+\mathcal{P}/L^+\mathcal{P}_s$  maps via \eqref{sympliso} isomorphically onto $\SP(E)/Q_p$, where $Q_p$ is the parabolic subgroup of $\SP(E)$ with Weyl group equal to the stabilizer of $\mu_p$ in $W_0$. We obtain a factorization
$$\begin{tikzpicture} 
\matrix(a)[matrix of math nodes, 
row sep=2.5em, column sep=2.5em, 
text height=1.5ex, text depth=0.45ex] 
{L^+\mathcal{P}/L^+\mathcal{Q}_p & \mathscr{G}_{2m-p}(E,\phi^{\sympl})\\ 
\SP(E)/Q_p\\}; 
\path[->](a-1-1) edge node[left] {$\cong$} (a-2-1); 
\path[->](a-2-1) edge node[below right] {$\cong$} (a-1-2); 
\path[->](a-1-1) edge node[above] {$j$} (a-1-2); 
\end{tikzpicture} $$
and hence $j$ is an isomorphism. Consider the $(2p+1)$-dimensional $k$-vector space $H=(e^{-\mu_p}\Lambda_m+ u^{-1}\Lambda_m^{\vee})/\Lambda_m'$. The reduction modulo $u$ of the form $u^2\phi$ induces an orthogonal form $\phi^{\ort}$ on $H$. Denote by $\mathscr{G}_p(H,\phi^{\ort})$ the finite-dimensional Gra§mannian of $p$-dimensional isotropic subspaces $H'$ of the orthogonal space $(H,\phi^{\ort})$. Then the morphism $\text{pr}_2$ is a $\mathscr{G}_p(H,\phi^{\ort})$-bundle, and thus $\widetilde{\mathcal{S}}_p$ is irreducible and smooth of dimension 
\[
(p(2m-p) - {p(p-1)\over 2})+(p(p+1) - {p(p+1)\over 2})=p(2m+1-p).
\]
This proves the proposition.
\end{proof}

\section{Applications to some local models of Shimura varieties}
We give an application of the resolution constructed in Theorem \ref{resolsimplycon} to the theory of local models of Shimura varieties in the case of a quasi-split unitary group with a special maximal parahoric level structure. In this case, the resolution is identified with an irreducible component of the special fiber of some model $\mathcal{M}^{\wedge}_s$. This is in analogy with the work of G. Pappas and M. Rapoport in \cite{PRsplitting} and N. Kr\"amer in \cite{Kraemer}. We deduce some geometric consequences for the local model which were obtained earlier by K. Arzdorf in \cite{Arzdorf} using direct computations. 

First we fix some notation. Let $L/L_0$ be a quadratic totally ramified extension of local fields with ring of integers $\mathcal{O}_L/\mathcal{O}_{L_0}$ and algebraically closed residue field $k$ of characteristic $\neq 2$. Fix a uniformizer $\pi_0\in\mathcal{O}_{L_0}$, and a uniformizer $\pi\in\mathcal{O}_L$ with $\pi^2=\pi_0$. Denote by $\bar{\cdot}$ the non-trivial element of $\Gal(L/L_0)$. Let $(W,\phi)$ be a $L/L_0$-hermitian space of odd dimension $n=2m+1\geq 3$. We define the non-degenerate symmetric $L_0$-bilinear form $(\str,\str):W\times W\rightarrow L_0$ by
\[
(x,y)\define {1\over 2}\text{Tr}(\phi(x,y)).
\]
For a $\mathcal{O}_L$-lattice $\lambda$, denote by $\lambda^{\vee}$ the dual with respect to $\phi$, and by $\hat{\lambda}$ the dual with respect to $(\str,\str)$. Then $\hat{\lambda}=\pi^{-1}\lambda^{\vee}$. Assume that $\phi$ splits, i.e. there is a basis $e_1,\ldots,e_n$ such that 
\[ 
 \phi(e_i,e_{n-j+1})=\delta_{i,j} \qquad \forall \;\; i,j=1,...,n.
\] 
For $i=0,\ldots,n-1$, define the $\mathcal{O}_L$-lattices
\[
\lambda_i\define\spanned_{\mathcal{O}_L}\{\pi^{-1}e_1,...,\pi^{-1}e_i,e_{i+1},...,e_n\},
\]
and complete $\{\lambda_0,\ldots,\lambda_{n-1}\}$ into a selfdual periodic lattice chain $\lambda_\bullet=\{\lambda_i\}_{i\in\mathbb{Z}}$. \\
For any $\mathcal{O}_L$-scheme $S$ and any $\mathcal{O}_L$-lattice $\lambda$, set $\lambda_S=\lambda\otimes_{\mathcal{O}_{L_0}}\mathcal{O}_S$ and likewise, write $(\str,\str)_S= (\str,\str)\otimes_{\mathcal{O}_{L_0}}\mathcal{O}_S$. Denote by $\Pi$ the operator $\pi\otimes 1$ on $\lambda_S$. Note that $\hat{\lambda}_{m,S}=\lambda_{m+1,S}$, i.e. $\lambda_{m,S}$ and $\lambda_{m+1,S}$ are in duality with respect to $(\str,\str)_S$.  

\subsection{The local model}
We follow \cite[\S1.e]{PRunitary} for the definition of the local model. See also \cite{Arzdorf} in this particular case. 

Fix non-negative integers $s<r$ with $r+s=n$. Let $M^{\naive}_s=M^{\naive}_{\{m\},s}$ be the functor on category of $\mathcal{O}_L$-schemes defined as follows: for any $\mathcal{O}_L$-scheme $S$, let $M^{\naive}_s(S)$ be the set of $\mathcal{O}_L\otimes_{\mathcal{O}_{L_0}}\mathcal{O}_S$-submodules 
\[
\mathcal{E}_S\subset \lambda_{m,S}, \qquad \mathcal{E}'_S\subset \lambda_{m+1,S}, 
\]
which are locally (on $S$) direct summands of rank $n$, subject to the conditions (N1)-(N3) below: \smallskip \\
(N1) There is a commutative diagram induced by the canonical lattice inclusions:
\[
\begin{tikzpicture} 
\matrix(a)[matrix of math nodes, 
row sep=2em, column sep=3.5em, 
text height=1.5ex, text depth=0.45ex] 
{ \lambda_{m,S} & \lambda_{m+1,S} & \pi^{-1}\lambda_{m,S}\\ 
\mathcal{E} & \mathcal{E}' & \pi^{-1}\mathcal{E}\\}; 
\path[->](a-1-1) edge node[above] {$i$} (a-1-2); 
\path[right hook->](a-2-1) edge (a-1-1); 
\path[right hook->](a-2-2) edge (a-1-2);
\path[->](a-2-1) edge (a-2-2); 
\path[->](a-1-2) edge node[above] {$j$} (a-1-3); 
\path[->](a-2-2) edge (a-2-3);
\path[right hook->](a-2-3) edge (a-1-3); 
\end{tikzpicture}
 \]
(N2) The vector bundle $\mathcal{E}'=\mathcal{E}^{\perp}$ is the orthogonal complement of $\mathcal{E}$ with respect to
\[
(\str,\str)_S:\lambda_{m,S}\times\lambda_{m+1,S}\longrightarrow\mathcal{O}_S.
\] 
(N3) The characteristic polynomial of $\Pi|{\mathcal{E}}$ is given by 
\[
\text{det}(T-\Pi|{\mathcal{E}})=(T-\pi)^s(T+\pi)^r \in \mathcal{O}_S[T].
\]
The functor $M_s^{\naive}$ is a closed subfunctor of a product of finite-dimensional Gra§mannians and hence representable by a projective $\mathcal{O}_L$-scheme. \\
We consider the closed subscheme $M^{\wedge}_s\subset M^{\naive}_s$ which to any $\mathcal{O}_L$-scheme $S$, associates the set of pairs $(\mathcal{E},\mathcal{E}')\in M^{\naive}_s(S)$ such that condition (W) below holds:\smallskip \\
(W) The exterior powers
\[
\bigwedge^{r+1}(\Pi-\pi|\mathcal{E})=0 \qquad\text{and}\qquad   \bigwedge^{s+1}(\Pi+\pi|\mathcal{E})=0,
\]
vanish, and the same holds true with $\mathcal{E}$ replaced by $\mathcal{E}'$.\\
The scheme $M^{\naive}_s$ is called the \emph{naive local model of signature $(r,s)$} (associated to the group of unitary similitudes $\GU(\phi)$ and $J=\{m\}$), and $M^{\wedge}_s$ is called the \emph{wedge local model of signature $(r,s)$}. 

\begin{lem}\label{genisowedgenaive}
On generic fibers $M^{\wedge}_{s,\eta}=M^{\naive}_{s,\eta}$, and these are isomorphic to the finite-dimensional Gra§mannian $\mathscr{G}_{s,n}(W)$ of $s$-dimensional subspaces of $W$.
\end{lem}

\begin{proof}
Let $S$ be a scheme over $L$. We have an isomorphism of $\mathcal{O}_S$-algebras
\begin{align*}
L\otimes_{L_0}\mathcal{O}_S & \overset{\cong}{\longrightarrow} \mathcal{O}_S\times\mathcal{O}_S\\
(x,y) & \longmapsto (xy,\bar{x}y).
\end{align*}
Let $(\mathcal{E},\mathcal{E}')\in M^{\naive}_{s,\eta}(S)$. Let $\mathcal{E}=\mathcal{E}^+\oplus\mathcal{E}^-$ be the decomposition according to the above isomorphism of $\mathcal{O}_S$-algebras. The operator $\Pi$ acts through $+\pi$ on $\mathcal{E}^+$ and through $-\pi$ on $\mathcal{E}^-$. By Condition (N3) the summand $\mathcal{E}^+$ has rank $s$, and $\mathcal{E}^-$ has rank $r$. The same is true for $\mathcal{E}'$, and hence Condition (W) is automatic, i.e. $M^{\wedge}_{s,\eta}(S)=M^{\naive}_{s,\eta}(S)$. By Condition (N2) the module $\mathcal{E}'$ is already determined by $\mathcal{E}$. On the other hand, $\mathcal{E}^-$ is already determined by $\mathcal{E}^+$, and we obtain an isomorphism
\begin{align*}
M^{\naive}_{s,\eta}(S) & \overset{\cong}{\longrightarrow} \mathscr{G}_{s,n}(W)\\
(\mathcal{E},\mathcal{E}') &\longmapsto \mathcal{E}^+.
\end{align*}
\end{proof}

There is a further variant: let $M^{\loc}_s$ be the scheme theoretic closure of the generic fiber $M^{\naive}_{s,\eta}$ in $M^{\naive}_s$. The scheme $M^{\loc}$ is called the \emph{local model of signature $(r,s)$}. We have closed immersions of projective $\mathcal{O}_L$-schemes
\[
M^{\loc}_s\subset M^{\wedge}_s\subset M^{\naive}_s,
\]
which are equalities on generic fibers.

\subsection{A candidate for a semistable resolution of $M^{\wedge}_s$} 
In analogy with \cite{PRsplitting} and \cite{Kraemer}, let $\mathcal{M}^{\wedge}_s$ be the functor which is defined as follows: for any $\mathcal{O}_L$-scheme $S$, let $\mathcal{M}^{\wedge}_s(S)$ be the set of triples $(\mathcal{E},\mathcal{E}',\mathcal{G}')$ with $(\mathcal{E},\mathcal{E}')\in M^{\wedge}_s(S)$ and a $\mathcal{O}_L\otimes_{\mathcal{O}_{L_0}}\mathcal{O}_S$-submodule 
\[
\mathcal{G}'\subset \mathcal{E}', 
\]
which is locally (on $S$) a direct summand of rank $s$, subject to the conditions (R1) and (R2) below:\smallskip \\
(R1) There is an inclusion $(\Pi+\pi)\mathcal{E}'\subset\mathcal{G}'$.\smallskip \\
(R2) The operator $(\Pi-\pi)|\mathcal{G}'= 0$ vanishes.

The functor is as a closed subfunctor of a product of finite-dimensional Gra§mannians representable by a projective $\mathcal{O}_L$-scheme. Denote by 
\begin{align}\label{semistableresmorph}
 \pi_s: \mathcal{M}^{\wedge}_s &\longrightarrow M^{\wedge}_s \\
 \notag (\mathcal{E},\mathcal{E}',\mathcal{G}') &\longmapsto (\mathcal{E},\mathcal{E}'),  
 \end{align}
 the canonical projection. The morphism $\pi_s$ is an isomorphism on generic fibers: With the notation from the proof of Lemma \ref{genisowedgenaive}, we have $\mathcal{G}'=(\mathcal{E}')^+$ on the generic fiber.\\
On the special fiber of the wedge local model $\overline{M}^{\wedge}_s=M^{\wedge}_s\otimes k$ only condition (W) above depends on $s$, and for $i<j$, there is natural closed immersion $\overline{M}^{\wedge}_i\subset\overline{M}^{\wedge}_j$. Hence, we obtain a chain of closed immersion 
\begin{equation*}
\{*\}=\overline{M}^{\wedge}_0\subset\overline{M}^{\wedge}_1\subset\ldots\subset\overline{M}^{\wedge}_{s-1}\subset\overline{M}^{\wedge}_s.
\end{equation*}
For $i=0,\ldots,s$, let $Z_i$ be the strict transform of the open subscheme\footnote{By definition $\overline{M}^{\wedge}_{-1}=\varnothing$ is the empty set.} $\overline{M}^{\wedge}_i\backslash\overline{M}^{\wedge}_{i-1}\subset \overline{M}^{\wedge}_i$ under the special fiber $\overline{\pi}_s:\overline{\mathcal{M}}^{\wedge}_s\rightarrow \overline{M}^{\wedge}_s$ of the morphism \eqref{semistableresmorph}.

\begin{thm}\label{locmodthm}
 \emph{(i)} The schemes $Z_0,\ldots,Z_s$ are irreducible and generically smooth of dimension $rs$, and hence the special fiber 
 $
 \overline{\mathcal{M}}^{\wedge}_s=\cup_{i=0}^sZ_i
 $
 is the union of $s+1$ divisors on the scheme $\mathcal{M}^{\wedge}_s$. The divisors $Z_0$, $Z_s$ are smooth, and the restriction $\overline{\pi}_s|_{Z_s}:Z_s\rightarrow \overline{M}^{\wedge}_s$ is a surjective birational projective morphism. In particular, $\overline{M}^{\wedge}_s$ is irreducible and contains a non-empty open reduced subscheme. \smallskip \\
 \emph{(ii)} The special fiber of the local model $\overline{M}^{\loc}_s$ is equal to the reduced locus $(\overline{M}^{\wedge}_s)_{\text{\rm red}}$, and hence $\overline{\pi}_s|_{Z_s}$ factors through $\overline{M}^{\loc}_s$ defining the morphism $\theta:Z_s\rightarrow \overline{M}^{\loc}_s$. Under a suitable\footnote{See \S \ref{lastsect} below.} embedding in a twisted affine Gra§mannian, the morphism $\theta$ is identical to the equivariant affine Demazure resolution in Proposition \ref{demresolunit}.
\end{thm}

\begin{rem}\label{semistable1}
(i) If $(r,s)=(n-1,1)$, an explicit calculation, analogous to the calculation in \cite{Kraemer}, shows that the morphism 
\[
\pi_1:\mathcal{M}^{\wedge}_1\longrightarrow M^{\wedge}_1
\]
is a semistable resolution, i.e. $\mathcal{M}_1^{\wedge}$ is regular, and the irreducible components of the special fiber $\{Z_i\}_{i=0,1}$  are smooth divisors crossing normally. However, in this case the local model $M^{\loc}_1$ is already smooth as was pointed out by the author \cite[Prop. 4.16]{Arzdorf}. \\ 
(ii) One may ask wether $\pi_s: \mathcal{M}^{\wedge}_s \rightarrow M^{\wedge}_s$ is a semistable resolution in general. I know of no counterexample.
\end{rem}

First we need a few lemmas on the structure of $\overline{\mathcal{M}}^{\wedge}_s$. Then we explain the embedding into the twisted affine Gra§mannian in \S \ref{lastsect} below and finish the proof of Theorem \ref{locmodthm}.\\
 Let $(\mathcal{E},\mathcal{E}',\mathcal{G}')$ be the universal triple over the special fiber $\overline{\mathcal{M}}^{\wedge}_s$. Define the (non-empty) open subset 
\[
\mathcal{U}^{\wedge}_s\define\{x\in\overline{\mathcal{M}}^{\wedge}_s\;|\; (\mathcal{G}'/\Pi\mathcal{E}')\otimes\kappa(x)=0\}.
\]
Since $\Pi\mathcal{E}'|{\mathcal{U}^{\wedge}_s}=\mathcal{G}'|{\mathcal{U}^{\wedge}_s}$,
the restriction of $\overline{\pi}_s:\overline{\mathcal{M}}^{\wedge}_s\rightarrow \overline{M}^{\wedge}_s$ to $\mathcal{U}^{\wedge}_s$ is an isomorphism onto an open subset $U^{\wedge}_s\subset \overline{M}^{\wedge}_s\backslash\overline{M}^{\wedge}_{s-1}$. We will see later that $U^{\wedge}_s= \overline{M}^{\wedge}_s\backslash\overline{M}^{\wedge}_{s-1}$.\\
Consider the $n$-dimensional $k$-vector $\Pi\lambda_{m+1,k}$, and its $2m$-dimensional subspace $i\Pi\lambda_{m,k}$ ($2m=n-1$), where $i$ denotes the morphism induced by the lattice inclusion $\lambda_m\subset\lambda_{m+1}$ (see Condition (N1) above). We equip $i\Pi\lambda_{m,k}$ with the non-degenerate symplectic form $\langle\str,\str\rangle:i\Pi\lambda_{m,k}\times i\Pi\lambda_{m,k}\rightarrow k$ defined by
\[
\langle i\Pi u, i\Pi v\rangle \define (u,i\Pi v)_k,
\]
for $u,v \in \lambda_{m,k}$. Note that this is well-defined. Denote by $\mathscr{G}_{s,2m}(i\Pi\lambda_{m,k},\langle\str,\str\rangle)$ the finite-dimensional Gra§mannian of $s$-dimensional isotropic subspaces of $i\Pi\lambda_{m,k}$ with respect to $\langle\str,\str\rangle$.\\ 
By Condition (R1) above, we have a morphism
\begin{align*}
\overline{\mathcal{M}}_s^{\wedge} & \longrightarrow\mathscr{G}_{s,n}(\Pi\lambda_{m+1,k}) \\
(\mathcal{E},\mathcal{E}',\mathcal{G}') & \longmapsto \mathcal{G}'.
\end{align*}
Define $Z$ by the cartesian diagram
\[
\begin{tikzpicture} 
\matrix(a)[matrix of math nodes, 
row sep=2em, column sep=3.5em, 
text height=1.5ex, text depth=0.45ex] 
{ Z & \overline{\mathcal{M}}_s^{\wedge}\\ 
\mathscr{G}_{s,2m}(i\Pi\lambda_{m,k},\langle\str,\str\rangle) & \mathscr{G}_{s,n}(\Pi\lambda_{m+1,k}).\\}; 
\path[right hook->](a-1-1) edge (a-1-2); 
\path[->](a-1-1) edge (a-2-1); 
\path[->](a-1-2) edge (a-2-2);
\path[right hook->](a-2-1) edge (a-2-2); 
\end{tikzpicture}
\]

\begin{lem}\label{stricttransform}
The scheme $Z$ is irreducible smooth projective of dimension $rs$, and is equal to the scheme theoretic closure of $\mathcal{U}^{\wedge}_s$ in $\overline{\mathcal{M}}^{\wedge}_s$. 
\end{lem}

\begin{proof}
Let $\mathcal{G}'_{\sympl}$ be the universal element of $\mathscr{G}_{s,2m}(i\Pi\lambda_{m,k},\langle\str,\str\rangle)$. We have
\[
Z(S)=\{(\mathcal{E}_S,\mathcal{E}'_S,(\mathcal{G}'_{\sympl})_S)\in \mathcal{M}^{\wedge}_s(S)\},
\]
for any $k$-scheme $S$. We consider the locally direct summand  $\mathcal{V}=i^{-1}\Pi^{-1}\mathcal{G}'_{\sympl}$ of $\lambda_{m,\mathscr{G}_{s,2m}(i\Pi\lambda_{m,k},\langle\str,\str\rangle)}$ of rank $n+s+1$. For a triple $(\mathcal{E}_S,\mathcal{E}'_S,(\mathcal{G}'_{\sympl})_S)\in Z(S)$, we claim that
\begin{equation}\label{inclkette}
\mathcal{V}_S^{\perp}\subset \mathcal{E}'_S\subset i\mathcal{V}_S,
\end{equation}
where $\mathcal{V}_S^{\perp}$ is the orthogonal complement of $\mathcal{V}_S$ with respect to $(\str,\str)_S$. \smallskip \\
\emph{Second inclusion in \ref{inclkette}:} consider the (degenerate) symmetric bilinear form
\[ 
(\str,i\str)_S: \lambda_{m,S}\times\lambda_{m,S}\longrightarrow \mathcal{O}_S. 
\]
We see that $\ker(i)=\langle\pi e_{m+1,S}\rangle\subset\mathcal{E}_S$. From Condition (N2), it follows $\mathcal{E}'_S\subset i\lambda_{m,S}$, and hence (using Condition (R1)) the second inclusion is obvious.\smallskip \\
\emph{First inclusion in \ref{inclkette}:} by taking complements, we obtain from the second inclusion $(i\mathcal{V}_S)^{\perp}\subset \mathcal{E}_S$. Since
\[
(\mathcal{V}_S,i(i\mathcal{V}_S)^{\perp})=((i\mathcal{V}_S)^{\perp},i\mathcal{V}_S)=0,
\]
we obtain $i(i\mathcal{V}_S)^{\perp}\subset\mathcal{V}_S^{\perp}$. This is an inclusion of locally direct summands of rank $n-s-1$ and is thus an equality. Then the first inclusion follows from  
\[
\mathcal{V}_S^{\perp}=i(i\mathcal{V}_S)^{\perp}\subset i\mathcal{E}_S 
\]
and from Condition (N1) above.\smallskip \\
Note that the isotropy condition on $\mathcal{G}'_{\sympl}$ translates into $\mathcal{G}'_{\sympl}\subset\mathcal{V}_S^{\perp}$. \\
Since $i(i\mathcal{V}_S)^{\perp}=\mathcal{V}_S^{\perp}$, the morphism $i$ induces an isomorphism
\[
\overline{i}: \mathcal{V}_S/(i\mathcal{V}_S)^{\perp}\overset{\cong}{\longrightarrow} i\mathcal{V}_S/\mathcal{V}_S^{\perp},
\]
and hence gives rise to a non-degenerate symmetric bilinear form
\[
(\str,\overline{i}\str)_S:\mathcal{V}_S/(i\mathcal{V}_S)^{\perp}\times\mathcal{V}_S/(i\mathcal{V}_S)^{\perp}\longrightarrow\mathcal{O}_S.
\] 
Note that $\mathcal{V}_S/(i\mathcal{V}_S)^{\perp}$ is a locally free $\mathcal{O}_S$-module of rank $2s+1$. The condition $i\mathcal{E}_S\subset\mathcal{E}_S'$ is equivalent to the condition
\[
\mathcal{E}_S/(i\mathcal{V}_S)^{\perp}\subset (\mathcal{E}_S/(i\mathcal{V}_S)^{\perp})^{\perp'},
\]
where $(\str)^{\perp'}$ denotes the complement with respect to $(\str,\overline{i}\str)_S$. All in all, we have constructed a morphism of $\mathscr{G}_{s,2m}(i\Pi\lambda_{m,k},\langle\str,\str\rangle)$-schemes
\begin{align*}
f: Z & \longrightarrow \mathscr{G}_{s,2s+1}(\mathcal{V}/(i\mathcal{V})^{\perp}, (\str,\overline{i}\str)) \\
(\mathcal{E},\mathcal{E}',\mathcal{G}'_{\sympl}) & \longmapsto \mathcal{E}/(i\mathcal{V})^{\perp}, 
\end{align*}
where $\mathscr{G}_{s,2s+1}(\mathcal{V}/(i\mathcal{V})^{\perp}, (\str,\overline{i}\str))$ is the finite-dimensional Gra§mannian of $s$-dimensional isotropic subspaces with respect to $(\str,\overline{i}\str)$. It is easy to see that $f$ is a projective monomorphism and hence is a closed immersion [EGA4, Prop. 8.11.5]. We claim that $f$ is an isomorphism. Since $\mathscr{G}_{s,2s+1}(\mathcal{V}/(i\mathcal{V})^{\perp}, (\str,\overline{i}\str))$ is irreducible smooth (in particular reduced) of dimension
\[
(s(2m-s) - {s(s-1)\over 2})+(s(s+1) - {s(s+1)\over 2})=s(n-s),  
\]
it is enough to prove that $\dim(Z)\geq s(n-s)$. To finish the proof, we show that $\mathcal{U}^{\wedge}_s\subset Z$, i.e. if $(\mathcal{E}_S,\mathcal{E}'_S,\mathcal{G}'_S)$ is a $S$-valued point of $\mathcal{M}^{\wedge}_s\otimes k$ with $\Pi\mathcal{E}'_S=\mathcal{G}'_S$, then $\mathcal{G}'_S$ is contained in $i\Pi\lambda_{m,S}$ and is isotropic with respect to $\langle\str,\str\rangle$: \\
We have seen that $\mathcal{E}'_S\subset i\lambda_{m,S}$ and thus $\mathcal{G}'_S=\Pi\mathcal{E}'_S\subset i\Pi\lambda_{m,S}$. It remains to show that $\langle\Pi\mathcal{E}'_S,\Pi\mathcal{E}'_S\rangle_S=0$. We may assume that $S=\Spec(R)$ is the spectrum of a local ring. Let $x\in S(k)$ be the closed point. The subspace $i\Pi\mathcal{E}_{\kappa(x)}\subset\Pi\mathcal{E}'_{\kappa(x)}$ is of codimension $\leq 1$. Hence, we can write $\Pi\mathcal{E}'_{\kappa(x)}=i\Pi\mathcal{E}_{\kappa(x)}+ \spanned\{\bar{v}\}$. Let $v$ be a lift of $\bar{v}$ in $\Pi\mathcal{E}'_S$. By Nakayama's Lemma $\Pi\mathcal{E}'_S=i\Pi\mathcal{E}_S+ \spanned\{v\}$ on $S$. But $\langle i\Pi\mathcal{E}_S,\Pi\mathcal{E}'_S\rangle_S=0$ by definition of $\langle\str,\str\rangle_S$, and we have to show that $\langle v,v\rangle_S=0$. But the form $\langle\str,\str\rangle_S$ is symplectic and $\chara(k)\neq 2$. 
\end{proof}

\begin{cor}\label{arzcor}
The morphism $\overline{\pi}_s|_{Z}:Z\longrightarrow \overline{M}^{\wedge}_s$ is surjective and birational. In particular, $\overline{M}^{\wedge}_s$ is irreducible and contains a non-empty open reduced subscheme.
\end{cor}

\begin{proof}
It is enough to show that $Z(k)\rightarrow \overline{M}^{\wedge}_s(k)$ is surjective, i.e. if $(\mathcal{E}_k,\mathcal{E}'_k)\in \overline{M}^{\wedge}_s(k)$, then there is a $\mathcal{G}'_k\in\mathscr{G}_{s,2m}(i\Pi\lambda_{m,k},\langle\str,\str\rangle)(k)$ such that $(\mathcal{E}_k,\mathcal{E}'_k,\mathcal{G}'_k)\in Z_s(k)$. So we have to find a $s$-dimensional isotropic subspace $\mathcal{G}'_k\subset i\Pi\lambda_{m,k}$ such that
\[
\Pi\mathcal{E}'_k\subset \mathcal{G}'\subset i\Pi\lambda_{m,k}\cap\mathcal{E}'_k. 
\]
In the proof of Proposition \ref{stricttransform} we have seen that $(\Pi\mathcal{E}'_k)^{\perp'}\subset \Pi\mathcal{E}'_k$, where $(\str)^{\perp'}$ denotes the orthogonal complement with respect to $\langle\str,\str\rangle$. We claim that $(\Pi\mathcal{E}'_k)^{\perp'}\subset i\Pi\lambda_{m,k}\cap\mathcal{E}'_k$, i.e. $(\Pi\mathcal{E}'_k)^{\perp'}\subset\mathcal{E}'_k$: let $i\Pi v\in i\Pi\lambda_{m,k}$ with $\langle i\Pi v, \Pi\mathcal{E}'_k\rangle=0$. In particular, we have 
\[
0=\langle i\Pi v, i\Pi\mathcal{E}_k\rangle=(v,i\Pi\mathcal{E}_k)=(\mathcal{E}_k,i\Pi v),
\]   
and hence $i\Pi v\in\mathcal{E}_k^{\perp}=\mathcal{E}'_k$. This proves the claim. The existence of $\mathcal{G}'$ is now obvious.
\end{proof}

\subsection{Embedding in the affine Gra§mannian}\label{lastsect}
A standard technique is the embedding of the special fiber of the local model in the corresponding twisted affine flag variety. For the case of a ramified unitary group, see \cite[\S 3.c]{PRunitary}.   

Let $k\rpot{u}/k\rpot{t}$ be an extension of local fields of Laurent series with $u^2=t$. We use the notation from \S  \ref{unitary}. Consider the the standard lattice chain $\Lambda_\bullet$. For $i\in\mathbb{Z}$, we fix isomorphisms compatible with $\pi\otimes 1$, resp. $u\otimes 1$
\[
\lambda_i\otimes_{\mathcal{O}_L}k\cong\Lambda_i\otimes_{k\pot{t}}k,
\]  
which sends the natural basis of each side to one another. For a $k$-algebra $R$ and a $R$-valued point $(\mathcal{E},\mathcal{E}')\in M^{\wedge}_s(R)$, define the $R\pot{u}$-lattice $\mathcal{L}_{\mathcal{E}}$ (resp. $\mathcal{L}_{\mathcal{E'}}$) as the inverse image of $\mathcal{E}$ (resp. $\mathcal{E}'$) under the canonical projection
\begin{align}\label{projection}
& \Lambda_{\{m\}}\otimes_{k\pot{t}}R\pot{t}\longrightarrow\Lambda_{\{m\}}\otimes_{k\pot{t}}R \\
\notag \text{(resp.}\;\;\; & \Lambda_{\{m+1\}}\otimes_{k\pot{t}}R\pot{t}\longrightarrow\Lambda_{\{m+1\}}\otimes_{k\pot{t}}R \; \text{).}
\end{align}
Let $\mathcal{F}=LGU(\phi)/L^+\mathcal{P}$ be the twisted affine Gra§mannian for the group of unitary similitudes $GU(\phi)$ and the parahoric group scheme $\mathcal{P}=\mathcal{P}_I$ corresponding to $I=\{m\}$ (cf. \cite[\S 3.c]{PRunitary}). Then we obtain a closed immersion into the neutral component 
\begin{equation}\label{embedding}
\begin{split}
\overline{M}^{\wedge}_s & \longrightarrow \mathcal{F}^0\\
(\mathcal{E},\mathcal{E}') &\longmapsto u^{-1}\mathcal{L}_{\mathcal{E}}(\subset u^{-1}\mathcal{L}_{\mathcal{E'}}),
\end{split}
\end{equation}
which is equivariant for the action of $L^+\mathcal{P}$. The point $(\Pi\lambda_{m,k},\Pi\lambda_{m+1,k})$ maps to the standard chain corresponding to $\Lambda_m$. As a corollary of Corollary \ref{arzcor}, we obtain the main result from \cite{Arzdorf}:

\begin{cor}\label{locmodred}
The special fiber of the local model $\overline{M}^{\loc}_s$ is via \eqref{embedding} identical to the Schubert variety $S_p$ defined in \S \ref{unitary} for $p=s$. Hence, it is normal, Frobenius split (if $\chara(k)>0$) and with only rational singularities. 
\end{cor}

\begin{proof}
The proof goes along the lines of \cite[Remark 11.4]{PRtwisted}: the embedding of $\overline{M}^{\wedge}_s$ into the affine Gra§mannian identifies $(\overline{M}^{\wedge}_s)_{\red}=(\overline{M}^{\loc}_s)_{\red}$ (Cor. \ref{arzcor}) with a Schubert variety in $\mathcal{F}^0$. By Theorem \ref{normality}, the Schubert variety $(\overline{M}^{\loc}_s)_{\red}$ is normal, Frobenius split and has only rational singularities. Since $M^{\loc}_s$ contains an open reduced subset (Cor. \ref{arzcor}), an application of Hironaka's lemma [EGA4, IV.5.12.8] to the flat $\mathcal{O}_L$-scheme $M^{\loc}_s$ shows that $\overline{M}^{\loc}_s=(\overline{M}^{\loc}_s)_{\red}$. In fact, $\overline{M}^{\loc}_s$ is identical to $S_p$ with $p=s$: the reduced neutral component $\mathcal{F}^0_{\text{red}}$ is isomophic to the corresponding twisted affine Gra§mannian for the special unitary group \cite[Prop. 6.6]{PRtwisted}.  Then the same calculation as in \cite[2.d.2]{PRunitary} shows that $\overline{M}^{\loc}_s$ is indeed identical to $S_p$.
\end{proof}

\begin{rem}
If $(r,s)=(n-1,1)$, the special fiber $\overline{M}^{\loc}_1$ is the quasi-minuscule Schubert variety in $\mathcal{F}^0$. But by Remark \ref{semistable1} above, $\overline{M}^{\loc}_1$ is smooth. Hence, the analogue of the Theorem of S. Evens and I. Mirkovi\'c \cite[Cor. B]{EvensMirkovic}, namely that the smooth locus of a Schubert variety is the open Schubert cell, does not hold for \emph{twisted} affine Gra§mannians. We will dicuss this question in a future paper.
\end{rem}

\begin{proof}[Proof of Theorem \ref{locmodthm}.]
First we prove (ii):\\
We will embed the irreducible smooth scheme $Z$ (cf. Lemma \ref{stricttransform}) in a product of twisted affine flag varieties compatible with the projection $\overline{\pi}_s|_{Z}:Z\rightarrow\overline{M}^{\wedge}_s$, the embedding \eqref{embedding} and equivariant for the action of $L^+\mathcal{P}$: \\
For a point $(\mathcal{E},\mathcal{E}',\mathcal{G}')\in Z(R)$, denote by $\mathcal{L}_{\mathcal{G}'}$ the $R\pot{u}$-lattice defined as the preimage of $\mathcal{G}'\subset\mathcal{E}'$ under the projection $\mathcal{L}_{\mathcal{E}'}\rightarrow\mathcal{E}'$ (cf. \eqref{projection}). Let $p=s$. Then condition (R2) implies that $\mathcal{G}'\subset \Pi\Lambda_{m+1,R}$, and it is easy to see that $u\mathcal{L}^{\vee}_{\mathcal{G}'}\subset\Lambda_m$ defines a point of $(L^+\mathcal{P}/L^+\mathcal{Q}_p)(R)$, where $\mathcal{Q}_p$ is the parahoric group scheme of $GU(\phi)$ corresponding to the set of simple reflections $\{s_1,\ldots,\hat{s}_p,\ldots,s_m\}$ as in \S \ref{unitary}. We obtain a closed immersion
\begin{equation}
\begin{split}
\notag Z &\longrightarrow L^+\mathcal{P}/L^+\mathcal{Q}_p\times\mathcal{F}^0 \\
(\mathcal{E},\mathcal{E}',\mathcal{G}') & \longmapsto (u\mathcal{L}^{\vee}_{\mathcal{G}'}\subset\Lambda_m,u^{-1}\mathcal{L}_{\mathcal{E}}),
\end{split}
\end{equation}
such that the diagram
\[
\begin{tikzpicture} 
\matrix(a)[matrix of math nodes, 
row sep=2.5em, column sep=3.5em, 
text height=2ex, text depth=0.45ex] 
{ Z & L^+\mathcal{P}/L^+\mathcal{Q}_p\times\mathcal{F}^0 \\ 
\overline{M}^{\wedge}_s & \mathcal{F}^0 \\}; 
\path[right hook->](a-1-1) edge (a-1-2); 
\path[->](a-1-1) edge node[left] {$\overline{\pi}_s|_Z$} (a-2-1); 
\path[->](a-1-2) edge node[left] {$\text{pr}_2$} (a-2-2);
\path[right hook->](a-2-1) edge (a-2-2); 
\end{tikzpicture}
\]
commutes and is equivariant for the action of $L^+\mathcal{P}$. Since $\overline{\pi}_s|_{Z}:Z\rightarrow\overline{M}^{\wedge}_s$ is $L^+\mathcal{P}$-equivariant, it is an isomorphism over the open set $U^{\wedge}_s=\overline{M}^{\wedge}_s\backslash\overline{M}^{\wedge}_{s-1}$. This implies $Z_s=Z$ is irreducible smooth and of dimension $rs$. Then $\overline{\pi}_s|_{Z_s}$ factors as a birational projective morphism $\theta: Z_s\rightarrow\overline{M}^{\loc}_s$ of algebraic varieties, and is by Corollary \ref{locmodred} and Proposition \ref{demresolunit} identical to the equivariant affine Demazure resolution $m:\widetilde{S}_p\rightarrow S_p$ from \S \ref{unitary}. This shows (ii).\\
In view of Corollary \ref{arzcor}, it remains to show that $ \overline{\mathcal{M}}^{\wedge}_s=\cup_{i=0}^sZ_i$ with $Z_i$ irreducible generically smooth of dimension $rs$. The scheme $Z_0$ is the finite-dimensional Gra§mannian $\mathscr{G}_{s,n}(\Pi\lambda_{m,k})$ and hence also irreducible smooth of dimension $rs$. For $i=1,\ldots,s-1$, we see that the morphism $\overline{\pi}_s^{-1}(U^{\wedge}_i)\rightarrow U^{\wedge}_i$ is a $\mathscr{G}_{(s-i),(n-2i)}(\ker(\Pi|\mathcal{E}')/\Pi\mathcal{E}')$-bundle and hence $Z_i$ is irreducible and generically smooth of dimension
\[
(s-i)(n-s-i) + i(n-i)=s(n-s).
\] 
This shows (i).
\end{proof}

\bibliographystyle{abbrv}

\end{document}